\documentclass[11pt]{article}
\usepackage{sychung2023}
\usepackage{mathtools}
\usepackage[style=trad-abbrv, doi=false, url=false, isbn=false, eprint=true
]{biblatex}
\title{Uniform bounds, zero separation and monotonicity for the regular Coulomb wave functions}

\author{Seok-Young Chung\footnote{Seok-Young.Chung@ucf.edu. Department of Mathematics, University of Central Florida,
4393 Andromeda Loop N., Orlando, FL 32816, USA.}}
\date{}

\begin{document}

\maketitle

{\bf Abstract.} 
This paper begins by deriving the uniform bounds for the regular Coulomb wave function $F_{\ell,\eta}$ and its derivative $F_{\ell,\eta}'$. We then examine detailed zero configurations of $F_{\ell,\eta}$ and $F_{\ell+1,\eta}$, extending insights into the earlier work that was restricted to $\ell>-3/2$. Our investigation also includes an analysis of the monotonicity of the zeros of $F_{\ell,\eta}$ with respect to parameters $\ell$ and $\eta$, respectively. Furthermore, we expand our exploration to associated orthogonal polynomials, as well as the functions involving both $F_{\ell,\eta}$ and $F_{\ell,\eta}'$. Finally, we explore the breakdown of the Sturm separation theorem by means of the zeros of associated orthogonal polynomials.

\medskip


\medskip


\section{Introduction and historical results}

As is customary, we introduce the regular and irregular Coulomb wave functions, denoted as $F_{\ell,\eta}$ and $G_{\ell,\eta}$, respectively. These functions serve as linearly independent solutions of the second order differential equation
\begin{equation}\label{ode}
    \frac{d^2 y}{dx^2} + \rb{ 1 - \frac{2\eta}{x} - \frac{\ell\rb{\ell+1}}{x^2} } y = 0.
\end{equation}
where $\eta\in \R$ is known as the Sommerfeld parameter, and $\ell \in \mathbb{N}\cup\cb{0}$ represents the angular momentum quantum number. In this paper, we extend the domain of $\ell$ to the real numbers $\R$, allowing for a more comprehensive and continuous analysis. These Coulomb wave functions are widely used in quantum mechanics, nuclear and atomic physics, etc. We refer to \cite{Hey2019,Michel2008,Miyazaki-K-C-I2001,Nishiyama2008,PengStarace2006} and references therein.

In particular, we shall focus on the regular Coulomb wave function $F_{\ell,\eta}(x)$ which is of the form
\begin{equation}\label{structure}
    F_{\ell,\eta}(x) = C_{\ell,\eta} x^{\ell+1} \phi_{\ell,\eta}(x),
\end{equation}
where the normalizing (Gamow) constant $C_{\ell,\eta}$ and the function $\phi_{\ell,\eta}(x)$ are defined as
\begin{align}
    C_{\ell,\eta} &= \frac{2^\ell e^{-\pi \eta /2}\vb{\Gamma\rb{\ell+1+i\eta}} }{\Gamma\rb{2\ell+2}}, \label{Gamow} \\
    \phi_{\ell,\eta}(x) &= e^{-ix} {}_1F_1\rb{ \ell+1-i\eta,\, 2\ell+2;\, 2ix }.  \label{phi}
\end{align}
In the expressions above, ${}_1F_1(a, b; x)$ represents the confluent hypergeometric function, defined as
\begin{equation*}
    {}_1F_1\rb{a, b; x} = \sum_{k=0}^\infty \frac{\rb{a}_k}{\rb{b}_k} \frac{x^k}{k!},
\end{equation*}
where $(a)_k$ denotes the Pochhammer symbol defined as $(a)_k = \Gamma(a+k)/\Gamma(a)$ for $k\ge0$, and $\Gamma(x)$ denotes the gamma function.

It is of particular note that the function $F_{\ell,\eta}$ includes the first kind Bessel function as a special case, in accordance with the relation (see \cite[(14.6.6)]{AbramowitzStegun1972} or \cite[\S 33.5]{Olver-L-B-C2010})
\begin{equation}\label{Bessel}
    F_{\nu-1/2,0}(x) = \sqrt{\frac{\pi x}{2}} J_{\nu}(x),
\end{equation}
provided that $\nu\ge 1/2$. On the basis of \eqref{Gamow}, \eqref{phi} and \cite[(13.6.1)]{AbramowitzStegun1972}, we have that for $\nu \in \R \setminus \cb{-1/2,-3/2,-5/2,\cdots}$ (or $\ell\in \R \setminus\cb{-1,-2,\cdots}$),
\begin{equation*}
    F_{\nu-1/2,0}(x) = \frac{|\Gamma(\nu+1/2)|}{\Gamma(\nu+1/2)} \sqrt{\frac{\pi x}{2}} J_\nu (x) = (-1)^{ \lfloor \nu+1/2 \rfloor } \sqrt{\frac{\pi x}{2}} J_\nu (x)
\end{equation*}
In the present case, it is known that the zeros of $F_{\nu-1/2,0}$ have symmetry around the origin. Moreover, an interlacing pattern emerges for positive zeros, as detailed in \cite[\S 15.22]{Watson1922}:
\begin{equation*}
    \begin{split}
        0< j_{\nu,1} < j_{\nu+1,1} < j_{\nu,2} < j_{\nu+1,2} < \cdots \quad (\nu>-1) \\
        0< j_{\nu+1,1} < j_{\nu,1} < j_{\nu+1,2} < j_{\nu,2} < \cdots \quad (\nu\le -1)
    \end{split}
\end{equation*}
Additionally, further investigations can be found in \cite{ChoChung2021,Palmai2013}.

While there has been extensive research on the zeros and bounds of the Bessel function (specifically when $\eta=0$), much less is understood about that of the regular Coulomb wave function $F_{\ell,\eta}$ for $\eta\ne 0$. Here, we list the key contributions to this topic.

The result due to {\v S}tampach and {\v S}{\v t}ov{\' i}{\v c}ek \cite[Proposition 13]{StampachStovicek2014} provides the theorem including the zero separation properties for $\phi_{\ell,\eta}$ (equivalently $F_{\ell,\eta}$):
\begin{varthm}{A}\label{thm:A}
    Let $\eta\in \R$, $\ell>-3/2$ with the exception that $\ell\ne -1$ if $\eta\ne 0$. Then the following statements hold
    \begin{enumerate}[label={\rm(\roman*)}]
    \item The zeros of the function
    $\phi_{\ell,\eta}$ form a countable subset of
    $\mathbb{R}\setminus\{0\}$ and do not accumulate at any finite point.
    \item All zeros of $\phi_{\ell,\eta}$ are simple, Moreover, the
    functions $\phi_{\ell,\eta}$ and $\phi_{\ell+1,\eta}$ have
    no zeros in common, and the zeros of
    $\phi_{\ell,\eta}$ and $\phi_{\ell+1,\eta}$ of the same sign are interlaced.
\end{enumerate}
\end{varthm}

Miyazaki et al. \cite[Remark 4.3]{Miyazaki-K-C-I2001} established a fundamental result stating that for $x>0,$ $\eta\in\mathbb{R}$, and $\ell\in \mathbb{N}$, there exists one and only one zero of $F_{\ell,\eta}'(x)$ between two consecutive zeros of $F_{\ell,\eta}(x)$.

Furthermore, Baricz \cite[Theorem 5]{Baricz2015} extended this result with the following theorem:
\begin{varthm}{B}\label{thm:B}
Let $\eta,\ell\in\mathbb{R}$. Then the following statements hold
\begin{enumerate}[label={\rm(\roman*)}]
    \item If $\ell>-1/2$, then the zeros of $F_{\ell,\eta}(x)$ and $F_{\ell,\eta}'(x)$ are interlacing.
    \item If $\ell>-1$, then the zeros of $F_{\ell,\eta}(x)$ and $x F_{\ell,\eta}'(x) -(\ell+1)F_{\ell,\eta}(x)$ are interlacing. 
\end{enumerate}
\end{varthm}

Concerning the zeros of $F_{\ell,\eta}$, Baricz and {\v S}tampach \cite[Theorem 11]{BariczStampach2018} established a Hurwitz-type theorem that elucidates the number of nonreal zeros as $\ell$ varies:
\begin{varthm}{C}\label{thm:C}
Let $\eta,\ell\in\R$. Then the following statements hold
\begin{enumerate}[label={\rm(\roman*)}]
    \item If $\ell >-3/2$, $\ell \ne -1$ if $\eta \ne 0$ and $\ell>-3/2$ if $\eta=0$, then $F_{\ell, \eta}$ has only real zeros.
    \item If $\ell <-3/2 $ with $\ell \not\in -\mathbb{N}/2$ if $\eta \neq 0$ and $\ell \notin -\mathbb{N}-\frac{1}{2}$ if $\eta = 0$, then $F_{\ell, \eta}$ has exactly $\lfloor -\ell - \frac{1}{2} \rfloor$ conjugate pairs of nonreal zeros with an infinite number of real zeros.
\end{enumerate}
\end{varthm}

The purpose of section 2 is to establish the uniform bounds for $F_{\ell,\eta}$ and its derivative $F_{\ell,\eta}'$ on the oscillatory region.
Section 3 is devoted to exploring the detailed distributions of zeros for $F_{\ell,\eta}$ and $F_{\ell+1,\eta}$. This section will also include an extended analysis of Theorem \ref{thm:A}, specifically focusing on the case when $\ell\le -3/2$. In section 4, we establish the monotonicity of the zeros of $F_{\ell,\eta}$ in terms of the parameters $\ell$ and $\eta$, respectively. In section 5, we shall investigate some properties of a sequence of orthogonal polynomials, which converges to $F_{\ell,\eta}$. Furthermore we examine the reality and interlacing of zeros in various functions that involve both $F_{\ell,\eta}$ and $F_{\ell,\eta}'$, which is closely related to Theorem \ref{thm:B}. Finally, the last section demonstrates the breakdown of the Sturm separation theorem, which was discussed in section 3, in connection with the zeros of associated orthogonal polynomials. In addition, section 4 and 5 provide the answers to the open problems, proposed in \cite[Open problem 1-3]{Baricz-C-D-T2016}.

\section{Uniform bounds for \texorpdfstring{$F_{\ell,\eta}$}{F ell eta} and \texorpdfstring{$F'_{\ell,\eta}$}{d F ell eta}}

To the best of our knowledge, the uniform bounds for the regular Coulomb wave functions $F_{\ell,\eta}$ have not been addressed, though uniform bounds for the Bessel functions (the case when $\eta =0$) have been derived by Landau \cite{Landau2000} and Krasikov \cite{Krasikov2006}. In details,
\begin{itemize}
    \item[--] (Landau) If $\nu>0$ and $x\in \R$, we have
    \begin{equation*}
        |J_\nu(x)| < b \nu^{-1/3},\quad |J_\nu(x)| < c|x|^{-1/3},
    \end{equation*}
    where $b = 0.674885\ldots$ and $c = 0.785746\ldots$ are the best possible constants.

    \item[--] (Krasikov) If $\nu>-1/2$ and $x > \sqrt{ \mu+\mu^{2/3} }/2$, we have
    \begin{equation*}
        |J_\nu(x)|^2 \le \frac{4(4x^2-(2\nu+1)(2\nu+5))}{\pi( (4x^2-\nu)^{3/2}+\mu) },
    \end{equation*}
    where $\mu=(2\nu+1)(2\nu+3)$.
\end{itemize}

We recall the function $\phi_{\ell,\eta} = x^{-\ell-1} F_{\ell,\eta}(x)/C_{\ell,\eta}$, presented in \eqref{phi}, which solves the second order differential equation
\begin{equation}\label{ode2}
    \phi_{\ell,\eta}''(x) + \frac{2(\ell+1)}{x}\phi_{\ell,\eta}'(x) + \rb{ 1 - \frac{2\eta}{x} } \phi_{\ell,\eta}(x) = 0,
\end{equation}
and it can be expressed in the form (see \cite[\S 14.1]{AbramowitzStegun1972})
\begin{equation}\label{series1}
     \phi_{\ell,\eta}(x) = \sum_{k=0}^\infty a_{\ell,\eta,k} x^{k},
\end{equation}
where the coefficients $\cb{a_{\ell,k}}_{k\ge0}$ are given by
\begin{equation}\label{series2}
    \begin{cases}
        \displaystyle a_{\ell,\eta,0} =1 ,\quad  a_{\ell,\eta,1} = \frac{\eta}{\ell+1},\\
        \displaystyle k(k+2\ell+1) a_{\ell,\eta,k} = 2\eta a_{\ell,\eta,k-1} - a_{\ell,\eta,k-2} \for k=2,\,3,\,\cdots.
    \end{cases}
\end{equation}
An alternative integral representation for $\phi_{\ell,\eta}$ can be found in \cite[\S 5]{Froberg1955}.

We also note that $\phi_{\ell,\eta}$ is real entire function of order of growth $1$ (see \cite[p. 262]{BariczStampach2018}) and by Theorem \ref{thm:C}, it has only real zeros, denoted by $\cb{x_{\ell,\eta,n}}_{n=1}^\infty$ provided $\eta\in \R$, $\ell >-3/2$ and $\ell\ne -1$ if $\eta\ne 0$. Additionally, the function $\phi_{\ell,\eta}$ admits Hadamard expansion (see \cite[(76)]{StampachStovicek2014}), given by
\begin{equation}\label{Hadamard}
    \phi_{\ell,\eta}(x) = \exp\rb{ \frac{\eta x}{\ell+1} } \prod_{n=1}^\infty \rb{ 1-\frac{x}{x_{\ell,\eta,n}} }e^{x/x_{\ell,\eta,n}}.
\end{equation}
Hence $\phi_{\ell,\eta}$ belongs to the Laguerre-P{\'o}lya class $\mathcal{LP}$ when $\eta\in \R$ and $\ell >-3/2$ with the exception that $\ell\ne -1$ if $\eta\ne 0$. Accordingly, the Laguerre inequality 
\begin{equation}\label{Laguerre0}
    L\qb{\phi_{\ell,\eta}}(x) \equiv \rb{\phi_{\ell,\eta}'(x)}^2 - \phi_{\ell,\eta}(x) \phi_{\ell,\eta}''(x) \ge0
\end{equation}
holds true under the same condition for $\eta$ and $\ell$. We refer to \cite{CsordasEscassut2005} for further information about the Laguerre-P{\'o}lya class $\mathcal{LP}$ and its properties, such as Laguerre inequality. In particular, by using \eqref{ode2}, we write
\begin{equation}\label{Laguerre1}
    L\qb{\phi_{\ell,\eta}}(x) = \frac{x^2 - 2\eta x -(\ell+1)^2}{x^2} \rb{\phi_{\ell,\eta}(x)}^2
    + \rb{ \phi_{\ell,\eta}'(x)+ \frac{\ell+1}{x} \phi_{\ell,\eta}(x) }^2,
\end{equation}
which implies that for $|x-\eta| > \sqrt{(\ell+1)^2+\eta^2}$,
\begin{equation}\label{Laguerre2}
    \begin{aligned}
        &\rb{\phi_{\ell,\eta}(x)}^2 \le \frac{x^2}{x^2 - 2\eta x -(\ell+1)^2}\, L\qb{\phi_{\ell,\eta}}(x),\\
        &\rb{ \phi_{\ell,\eta}'(x)+ \frac{\ell+1}{x} \phi_{\ell,\eta}(x) }^2 \le  L\qb{\phi_{\ell,\eta}}(x).
    \end{aligned}
\end{equation}

We now establish the lower and upper bounds for the Laguerre expression $L\qb{\phi_{\ell,\eta}}$ for given $\eta$ and $\ell$.

\begin{lemma}\label{lem:2.1}
Let $\eta\in \R$, $\ell >-3/2$ with the exception that $\ell\ne -1$ if $\eta\ne 0$. Then we have
    \begin{equation*}
        \frac{x-\eta -\sqrt{(\ell+1)^2+\eta^2}}{C_{\ell,\eta}^2 x^{2\ell+3}} < L\qb{\phi_{\ell,\eta}}(x) < \frac{x-\eta +\sqrt{(\ell+1)^2+\eta^2}}{C_{\ell,\eta}^2 x^{2\ell+3}},
    \end{equation*}
    for $x>0$.
\end{lemma}

\begin{proof}
Let us introduce auxiliary functions
\begin{equation*}
    f(x) =x^{2\ell+3}L\qb{\phi_{\ell,\eta}}(x), \quad g^{\pm}(x) = x-\eta \pm\sqrt{(\ell+1)^2+\eta^2}.
\end{equation*}
By eliminating the higher-order derivatives $\phi_{\ell,\eta}^{(n)}$, $n\ge 2$ with the aid of \eqref{ode2}, we observe that
\begin{align*}
&x^{-2\ell-2}W[g^+,f](x) \\
&= \rb{\sqrt{(\ell+1)^2+\eta^2}+\eta} \big(\phi_{\ell,\eta}(x)\big)^2\\
&\pushright{-2(\ell+1) \phi_{\ell,\eta}(x)\phi_{\ell,\eta}'(x)  +\rb{\sqrt{(\ell+1)^2+\eta^2}-\eta} \big(\phi_{\ell,\eta}'(x)\big)^2}\\
& =  \qb{ \rb{\sqrt{(\ell+1)^2+\eta^2}+\eta}^{1/2}\phi_{\ell,\eta}(x) - \rb{\sqrt{(\ell+1)^2+\eta^2}-\eta}^{1/2}\phi_{\ell,\eta}'(x) }^2,
\end{align*}
where $W[u,v](x) = u(x) v'(x) - u'(x) v(x)$ denotes the Wronskian.
Thereby, $f(x)/g^+(x)$ is increasing on $(0,\infty)$, as shown by
\begin{equation}\label{lem1}
    \frac{d}{dx} \frac{f(x)}{g^+(x)} =  \frac{W[g^+,f](x)}{\rb{x-\eta +\sqrt{(\ell+1)^2+\eta^2}}^2} \ge0.
\end{equation}
Owing to the asymptotic behavior (see for instance \cite[\S 33.10]{Olver-L-B-C2010})
\begin{align*}
    F_{\ell,\eta}(x) &= \sin\rb{ \theta_{\ell,\eta}(x) } +o(1),\quad \text{as } x\to \infty,
\end{align*}
where $\theta_{\ell,\eta}(x) = x-\eta \ln\rb{2x} -\frac{\ell}{2}\pi +\sigma_{\ell,\eta}$ and $\sigma_{\ell,\eta}$ denotes Coulomb phase shift defined as $\arg\rb{\Gamma\rb{\ell+1+i\eta}}$, it is not difficult to show that
\begin{equation}\label{lem2}
    \lim_{x\to \infty} \frac{f(x)}{g^+(x)} = \lim_{x\to \infty} \frac{x^{2\ell+3}L(\phi_{\ell,\eta})}{x-\eta +\sqrt{(\ell+1)^2+\eta^2}} =\frac{1}{C_{\ell,\eta}^2}.
\end{equation}
Hence on combining \eqref{lem1} and \eqref{lem2}, the right inequality has been established.

In a similar manner, we find that
\begin{multline*}
    x^{-2\ell-2}W[g^-,f](x) \\
    =-\qb{ \rb{\sqrt{(\ell+1)^2+\eta^2}-\eta}^{1/2}\phi_{\ell,\eta}(x) + \rb{\sqrt{(\ell+1)^2+\eta^2}+\eta}^{1/2}\phi_{\ell,\eta}'(x)  }^2.
\end{multline*}
Therefore, the left inequality follows since the function $f(x)/g^-(x)$
decreases monotonically for $x>\eta + \sqrt{(\ell+1)^2+\eta^2}$ and approaches to $1/C_{\ell,\eta}^2$ as $x$ tends to infinity. If $0<x\le \eta + \sqrt{(\ell+1)^2+\eta^2}$, the result is trivial, using \eqref{Laguerre0}.
\end{proof}

It is worth to note that Miyazaki \cite[Remark 4.1]{Miyazaki-K-C-I2001} proved that the region of positive zeros of $F_{\ell,\eta}(x)$ and $F_{\ell,\eta}'(x)$ is bounded by the inequality
\begin{equation}\label{bound}
    x > \eta + \sqrt{\eta^2 +(\ell+1)^2}
\end{equation}
for $\eta\in \R$ and $\ell=0,1,2,\cdots$. The range of parameter $\ell$ in this result can be extended as follows:

\begin{lemma}\label{lem:2.2}
Let $\eta\in \R$, $\ell >-3/2$ with the exception that $\ell\ne -1$ if $\eta\ne 0$.
The positive zeros of $F_{\ell,\eta}'$ are bounded below by $\eta + \sqrt{\eta^2+(\ell+1)^2}>0$.
\end{lemma}

\begin{proof}
Let $\beta$ be an arbitrary positive zero of $F_{\ell,\eta}'$, equivalently $\phi'_{\ell,\eta}(\beta) + (\ell+1)\phi_{\ell,\eta}(\beta)/\beta =0$. From \eqref{Laguerre0} and \eqref{Laguerre1}, we deduce that
\begin{equation*}
    \frac{\beta^2 - 2\eta \beta -(\ell+1)^2}{\beta^2} \rb{\phi_{\ell,\eta}(\beta)}^2 \ge 0.
\end{equation*}
Considering \cite[Theorem 2.1]{CsordasEscassut2005}, the equality holds only when $\phi_{\ell,\eta}$ has a multiple zero. On the other hand, by (ii) of Theorem \ref{thm:A}, it follows that $\phi_{\ell,\eta}$ has only simple zeros and further that $\phi_{\ell,\eta}(\beta)\ne 0$. Hence $\beta_2 -2\eta \beta - (\ell+1)^2 >0$, leading to the desired result
\begin{equation*}
    \beta > \eta + \sqrt{\eta^2 + (\ell+1)^2}.
\end{equation*}
\end{proof}
\noindent
We note that a similar result can be deduced for $F_{\ell,\eta}$, which will be discussed in section 5.

Combining \eqref{Laguerre2}, Lemma \ref{lem:2.1} and \ref{lem:2.2}, the result for uniform bounds can be established as follows:
\begin{theorem}\label{thm:2.1}
Let $\eta\in \R$, $\ell >-3/2$ with the exception that $\ell\ne -1$ if $\eta\ne 0$.Then the following inequalities hold
\begin{align*}
    \vb{F_{\ell,\eta}(x)}^2 &< \frac{x}{x - \eta  -\sqrt{(\ell+1)^2+\eta^2}},\\
    \vb{F_{\ell,\eta}'(x)}^2 &< 1+\frac{\sqrt{(\ell+1)^2+\eta^2}-\eta}{x},
\end{align*}
for $x > \eta + \sqrt{(\ell+1)^2+\eta^2}>0$. In particular, we have
\begin{align*}
    \vb{F_{\ell,\eta}(\beta)}^2 &> \frac{\beta}{\beta - \eta  +\sqrt{(\ell+1)^2+\eta^2}},\\
    \vb{F_{\ell,\eta}'(\alpha)}^2 &> 1-\frac{\sqrt{(\ell+1)^2+\eta^2}+\eta}{\alpha},
\end{align*}
where $\alpha,\beta$ denote any positive zeros of $F_{\ell,\eta},F'_{\ell,\eta}$, respectively.
\end{theorem}

\begin{proof}
Regarding \eqref{structure} and \eqref{Laguerre1}, we deduce that for $x > \eta + \sqrt{(\ell+1)^2+\eta^2}$,
\begin{align*}
    |F_{\ell,\eta}(x)|^2 &\le \frac{C_{\ell,\eta}^2\, x^{2\ell+4}}{x^2 - 2\eta x -(\ell+1)^2}\, L\qb{\phi_{\ell,\eta}}(x),\\
    |F_{\ell,\eta}'(x)|^2 &\le C_{\ell,\eta}^2\, x^{2\ell+2} L\qb{\phi_{\ell,\eta}}(x).
\end{align*}
The proof is straightforward by applying Lemma \ref{lem:2.1}.

To verify the reversed inequality, we let $\beta$ be any positive zero of $F_{\ell,\eta}'$. Since $F'_{\ell,\eta}(x) = x^{\ell+1}(\phi_{\ell,\eta}'(x) +(\ell+1) \phi_{\ell,\eta}(x)/x)$, we obtain
\begin{equation*}
    |F_{\ell,\eta}(\beta)|^2 = \frac{C_{\ell,\eta}^2\, \beta^{2\ell+4}}{\beta^2 - 2\eta \beta-(\ell+1)^2}\, L\qb{\phi_{\ell,\eta}}(\beta).
\end{equation*}
We note that Lemma \ref{lem:2.2} implies that $\beta^2 - 2\eta \beta - (\ell+1)^2 >0$.
Hence the left inequality in Lemma \ref{lem:2.1} gives the lower bound for $|F_{\ell,\eta}(\beta)|^2$. Similarly, we have that for any positive zero $\alpha$ of $F_{\ell,\eta}$,
\begin{equation*}
    |F_{\ell,\eta}'(\alpha)|^2 = C_{\ell,\eta}^2\, \alpha^{2\ell+4} L\qb{\phi_{\ell,\eta}}(\alpha).
\end{equation*}
The same argument can be applied to establish the lower bound for $|F_{\ell,\eta}'(\alpha)|^2$.
\end{proof}

\begin{remark}\
\begin{enumerate}[label=(\roman*)]
    \item The uniform bound can also be obtained in the exceptional case where $\ell=-1$ and $\eta\ne 0$, despite Theorem \ref{thm:2.1} omits this case because $F_{-1,\eta}$ fails to be in $\mathcal{LP}$.
    To be precise, in consideration of \cite[6.7.1 (12)]{Erdelyi-M-O-T1953} and \eqref{structure}-\eqref{phi}, it follows that for $\eta\ne0$, 
    \begin{equation*}
        F_{-1,\eta}(x) = \lim_{\ell\to-1} C_{\ell,\eta} x^{\ell+1} \phi_{\ell,\eta}(x) = -\frac{i\eta x}{2|\eta|} F_{0,\eta}(x),
    \end{equation*}
    which leads that for $x> \eta+ \sqrt{1+\eta^2}$ and $\eta\ne0$,
    \begin{equation*}
        \vb{F_{-1,\eta}(x)}^2 < \frac{x^3}{4(x - \eta  -\sqrt{1+\eta^2})}.
    \end{equation*}

    \item When $\eta=0$, Theorem \ref{thm:2.1}, along with \eqref{Bessel}, yields that
    \begin{equation*}
        |J_\nu(x)|^2 \le \frac{2}{\pi (x-|\nu-1/2|)}, 
    \end{equation*}
    for $\nu > -1$ and $x>|\nu-1/2|$.

    \item In order to address the bound on the negative real line, we recall reflection formula \cite[(23)]{Gaspard2018}
    \begin{equation*}
        F_{\ell,\eta}(-x) = - e^{\pi (-\eta + i\ell)} F_{-\ell,\eta}(x)
    \end{equation*}
    for $x>0$. Thus we have that for $x> \eta + \sqrt{ \eta^2 +(\ell+1)}$,
    \begin{equation*}
        \vb{F_{\ell,\eta}(-x)}^2 < \frac{x e^{-\pi \eta} }{x + \eta  -\sqrt{(\ell+1)^2+\eta^2}}.
    \end{equation*}

\end{enumerate}
\end{remark}

\section{Sturm separation theorem with \texorpdfstring{$F_{\ell,\eta}$}{phi ell eta} and \texorpdfstring{$F_{\ell+1,\eta}$}{phi ell+1 eta}}

As discussed in the preceding section, the case of $\eta=0$ corresponds to the Bessel function, extensively studied in literature. Throughout this paper, our attention is directed toward the scenario where $\eta\ne0$.

While the function $\phi_{\ell,\eta}(x)$ is entire in both $x$ and $\eta$, it has poles and becomes undefined as a function of $\ell$ whenever $2\ell+2$ coincides with non-negative integers.
To address this, we consider the \emph{modified} regular Coulomb wave functions $\varphi_{\ell,\eta}(x)$, defined as
\begin{equation*}
\varphi_{\ell,\eta}(x) = \frac{1}{\Gamma(2\ell+2)} \phi_{\ell,\eta}(x),
\end{equation*}
which admits the limiting case (see for instance \cite[\S 6.7]{Erdelyi-M-O-T1953} and \cite[p. 6]{Gaspard2018})
\begin{equation}\label{limit1}
    \lim_{\ell \to -(n+1)/2} \varphi_{\ell,\eta}(x) = \frac{ \Gamma((n+1)/2-i\eta) }{ \Gamma((-n+1)/2-i\eta) } x^n \varphi_{(n-1)/2,\eta}(x)
\end{equation}
for each $n\in \mathbb{N}$. As readily seen, the coefficient $\Gamma((n+1)/2-i\eta)/\Gamma((-n+1)/2-i\eta)$ can be simply expressed as
\begin{equation}\label{limit2}
    \begin{cases}
        (-1)^k\prod_{m=1}^k \rb{ \eta^2 + \rb{m-\frac{1}{2}}^2 }, & n=2k,\ k\in \mathbb{N}\\
        (-1)^k i\eta \prod_{m=1}^{k-1} \rb{ \eta^2 + m^2 }, & n=2k-1, \ k\in \mathbb{N},
    \end{cases}
\end{equation}
where an empty product is considered as $1$.
Given that $\varphi_{\ell,\eta}$ shares zeros with $F_{\ell,\eta}$, our focus will be on the investigation of $\varphi_{\ell,\eta}$ and $\varphi_{\ell+1,\eta}$ instead of $F_{\ell,\eta}$ and $F_{\ell+1,\eta}$.

In accordance with Sturm's oscillation theorem (see \cite[p. 53]{Al-Gwaiz2008}), it is evident that $F_{\ell,\eta}$ has a countable number of positive and negative zeros.
Let $\big\{\rho_{\ell,\eta,k}\big\}_{k=1}^\infty$ denote the sequence of all positive zeros of $F_{\ell,\eta}$, arranged in ascending order of magnitude. On applying the reflection formula \cite[(22)]{Gaspard2018}
\begin{equation}\label{reflect1}
    \varphi_{\ell,-\eta}(x) = \varphi_{\ell,\eta}(-x),
\end{equation}
the negative zeros of $\varphi_{\ell,\eta}$ correspond to the positive zeros of $\varphi_{\ell,-\eta}$. Thus the real zeros $\cb{x_{\ell,\eta,n}}_{n=1}^\infty$, introduced in section 2, can be denoted as $ \cb{\rho_{\ell,\eta,n}}_{n=1}^\infty \cup \cb{-\rho_{\ell,-\eta,n}}_{n=1}^\infty$. In this regard, it suffices to study the distribution of positive zeros of $F_{\ell,\eta}$. We note that the sequences $\cb{x_{\ell,\eta,n}}_{n=1}^\infty$ and $\cb{\rho_{\ell,\eta,n}}_{n=1}^\infty$ are used respectively, depending on the situation.

We proceed to establish the linearly independence of the functions $\varphi_{\ell,\eta}$ and $x\varphi_{\ell+1,\eta}$ for $\ell \in \rb{-3/2,\infty}\setminus\cb{-1}$. When $\ell=-1$, it follows from \eqref{limit1} and $\eqref{limit2}$ that $\varphi_{-1,\eta}(x) = -i\eta x\varphi_{0,\eta}(x)$. Thus, in this specific case, those functions become linearly dependent.

On the other hand, a classical result due to Wimp \cite[p. 892]{Wimp1985} provides an expression for the ratio of $\varphi_{\ell,\eta}$ and $\varphi_{\ell+1,\eta}$, for $\ell \in \rb{-3/2,\infty}\setminus\cb{-1}$ and $\eta \ne 0$, as follows:
\begin{equation}\label{ML1}
    \frac{x \varphi_{\ell+1,\eta}(x)}{\varphi_{\ell,\eta}(x)} = \frac{\ell+1}{2\rb{(\ell+1)^2+\eta^2}} \sum_{k=1}^\infty \frac{x}{x_{\ell,\eta,k}\big( x_{\ell,\eta,k}-x \big)}.
\end{equation}
By differentiating both sides, we find that under the same conditions of parameters,
\begin{equation}\label{Wrons1}
    W\qb{ \varphi_{\ell,\eta}(x), x \varphi_{\ell+1,\eta}(x) } = \frac{\ell+1}{2\rb{(\ell+1)^2+\eta^2}}(\varphi_{\ell,\eta}(x))^2\\
     \sum_{k=1}^\infty \frac{1}{\big( x- x_{\ell,\eta,k} \big)^2}.
\end{equation}
It is apparent that Wronskian $W\qb{ \varphi_{\ell,\eta}(x), x \varphi_{\ell+1,\eta}(x) }$ has removable singularities at $x= x_{\ell,\eta,k}$,
$k\ge1$, and thus it keeps a constant sign on $\R$ for given $\ell\in (-3/2,\infty)\setminus\cb{-1}$ and $\eta\ne0$. Consequently, $\varphi_{\ell,\eta}(x)$ and $x\varphi_{\ell+1,\eta}(x)$ form a fundamental set of solutions of the second order ODE, given by
\begin{equation*}
    \dm{ y'' & y' & y \\ \varphi_{\ell,\eta}'' & \varphi_{\ell,\eta}' & \varphi_{\ell,\eta} \\ (x\varphi_{\ell+1,\eta})'' & (x\varphi_{\ell+1,\eta})' & x\varphi_{\ell+1,\eta} } =0.
\end{equation*}
Thus, according to the Sturm separation theorem (see \cite[Theorem 2.8]{Al-Gwaiz2008}), it is immediate that $\varphi_{\ell,\eta}(x)$ has one and only one zero between any pair of consecutive zeros of $x\varphi_{\ell+1,\eta}(x)$, and vice versa, \rm{i.e.},
\begin{theorem}\label{thm:3.1}
    Let $\eta,\ell\in \R$ with $\ell >-3/2$, $\ell\ne-1$ and $\eta\ne 0$. Then the separation property for the zeros of $\varphi_{\ell,\eta}(x)$ and $x\varphi_{\ell+1,\eta}(x)$ holds according to the following pattern{\rm :}
    \begin{equation*}
        0< \rho_{\ell,\eta,1} < \rho_{\ell+1,\eta,1} < \rho_{\ell,\eta,2} < \cdots.
    \end{equation*}
\end{theorem}
In addition, an interlacing pattern for the negative zeros of $\varphi_{\ell,\eta}$ comes directly from \eqref{reflect1}.

\begin{remark}\label{rem:3.1} \
    \begin{enumerate}[label=(\roman*)]
        \item This result has been initially addressed by {\v S}tampach and {\v S}{\v t}ov{\' i}{\v c}ek \cite{StampachStovicek2014}, as presented in Theorem \ref{thm:A} within the framework of $\phi_{\ell,\eta}$ and $\phi_{\ell+1,\eta}$.
        \item Theorem \ref{thm:C} provides that $\varphi_{\ell,\eta}(x)$ and $x\varphi_{\ell+1,\eta}(x)$ have only real zeros for $\ell$ satisfying $\ell >-3/2$ and $\ell\ne -1$.

        \item By making use of the non-vanishing property for $W\qb{ \varphi_{\ell,\eta}(x), x \varphi_{\ell+1,\eta}(x) }$, one can readily confirm that the zeros of $\varphi_{\ell,\eta}(x)$ are all simple, as the Wronskian vanishes at zero with multiplicity exceeding $1$. 
    \end{enumerate}
\end{remark}

In contrast, when $\ell\le -3/2$, the formula \eqref{Wrons1} is no longer available and the Wronskian $W\qb{ \varphi_{\ell,\eta}(x), x \varphi_{\ell+1,\eta}(x) }$ actually has at least two zeros (counting multiplicity) on the real line, leading to the breakdown of the separation property for the zeros of $\varphi_{\ell,\eta}(x)$ and those of $x\varphi_{\ell+1,\eta}(x)$ on $\R$. Remarkably, despite this breakdown, $\varphi_{\ell,\eta}(x)$ and $\varphi_{\ell+1,\eta}(x)$ maintain an interlacing property within each of the intervals $\big( {-\infty}, 0\big)$ and $\big(0,\infty\big)$. Our objective is to demonstrate how the configuration of zeros changes in the scenario where $\ell \le -3/2$ and $\eta \ne 0$.

We recall the several relations of $F_{\ell,\eta}$ (see \cite[p. 539]{AbramowitzStegun1972}) as follows:
\begin{equation}\label{recur1}
    \begin{split}
        (\ell+1)F_{\ell,\eta}'(x) &= \rb{ \frac{(\ell+1)^2}{x}+\eta }F_{\ell,\eta}(x) - \sqrt{ (\ell+1)^2+\eta^2 } F_{\ell+1,\eta}(x),\\
        \ell\, F_{\ell,\eta}'(x) &= -\rb{ \frac{\ell^2}{x}+\eta }F_{\ell,\eta}(x) + \sqrt{ \ell^2 + \eta^2 } F_{\ell-1,\eta}(x).
    \end{split}
\end{equation}
\begin{equation}\label{recur2}
    \begin{split}
        \rb{\ell+1}\sqrt{\ell^2+\eta^2} F_{\ell-1,\eta}(x) - \rb{2\ell+1}\rb{\eta + \frac{\ell(\ell+1)}{x}} F_{\ell,\eta}(x) \\
        \pushright{+ \ell \sqrt{(\ell+1)^2+\eta^2} F_{\ell+1,\eta}(x)= 0.}
    \end{split}
\end{equation}

Let $\eta,\ell\in \R$ with $\ell \ne -1$, $\eta\ne 0$, and let us introduce two auxiliary functions, given by
\begin{equation}\label{UV}
    U_\ell(x)= x^{-(\ell+1)} e^{-\frac{\eta}{\ell+1}x} F_{\ell,\eta}(x),\quad V_{\ell+1}(x)= x^{\ell+1} e^{\frac{\eta}{\ell+1}x} F_{\ell+1,\eta}(x).
\end{equation}
Throughout this paper, the principal branch is chosen such that $-\pi < \text{arg}(x) \le \pi$.

\begin{lemma}\label{lem:3.1}
    Let $\eta,\ell\in \R$ with $\ell \ne -1$, $\eta\ne 0$. Then we have
    \begin{equation}\label{Wron1}
        W\qb{ U_\ell,\, V_{\ell+1} }(x) = \frac{\sqrt{ (\ell+1)^2 + \eta^2 }}{\ell+1}\big( F_{\ell,\eta}^2(x) + F_{\ell+1,\eta}^2(x) \big).
    \end{equation}
    Moreover, $\varphi_{\ell,\eta}$ and $\varphi_{\ell+1,\eta}$ cannot have common zeros except the origin. Consequently, $(\ell+1)x^{-2(\ell+1)}W\qb{ U_\ell,\, V_{\ell+1} }(x)>0$ for $x \in \R \setminus \cb{0}$.
\end{lemma}

\begin{proof}
    The relation \eqref{recur1} can be reformulated as
    \begin{align*}
        &U_\ell'(x) = - \frac{\sqrt{ (\ell+1)^2 + \eta^2 }}{\ell+1}\ x^{-(\ell+1)} e^{-\frac{\eta}{\ell+1}x} F_{\ell+1,\eta}(x),\\
        &V_{\ell+1}'(x) = \frac{\sqrt{ (\ell+1)^2 + \eta^2 }}{\ell+1}\ x^{\ell+1} e^{\frac{\eta}{\ell+1}x} F_{\ell,\eta}(x).
    \end{align*}
    Then the subsequent deduction yields the expression
    \begin{equation*}
        W\qb{ U_\ell,\, V_{\ell+1} }(x) = \frac{\sqrt{ (\ell+1)^2 + \eta^2 }}{\ell+1}\big( F_{\ell,\eta}^2(x) + F_{\ell+1,\eta}^2(x) \big).
    \end{equation*}
    Equivalently, this can be expressed as
    \begin{multline*}
        h_{\ell,\eta}(x) = 2^{2\ell} e^{-\pi\eta}\vb{ \Gamma(\ell+1+i\eta) }^2 \sqrt{ (\ell+1)^2 + \eta^2 }\\
        \cdot \Big( \varphi_{\ell,\eta}^2(x) + 4\big((\ell+1)^2 +\eta^2 \big)x^2 \varphi_{\ell+1,\eta}^2(x)  \Big),
    \end{multline*}
    where $h_{\ell,\eta}(x) = (\ell+1)x^{-2(\ell+1)}W\qb{ U_\ell,\, V_{\ell+1} }(x)$. Thus $h_{\ell,\eta}$ remains positive except at the origin and common zeros of $\varphi_{\ell,\eta}$ and $\varphi_{\ell+1,\eta}$.
    
    Suppose, on the contrary, that $\zeta \ne 0$ is any common zero of $\varphi_{\ell,\eta}$ and $\varphi_{\ell+1,\eta}$. Then it follows from the recurrence \eqref{recur1} that $\varphi_{\ell,\eta}'\rb{\zeta} = 0$. 
    In addition, the equation \eqref{ode} leads that $\varphi_{\ell,\eta}^{(n)}\rb{\zeta} = 0$ for all $n\ge0$. Thus the Taylor expansion for $\varphi_{\ell,\eta}(x)$ centered at $x=\zeta$ must be identically zeros, which is a contradiction. Therefore, $\varphi_{\ell,\eta}$ and $\varphi_{\ell+1,\eta}$ cannot share zeros, and thus $h_{\ell,\eta}(x)>0$ for $x\ne 0$.
\end{proof}

On making use of the above lemma, we elucidate the pattern on the zeros of $\varphi_{\ell,\eta}$ and $\varphi_{\ell+1,\eta}$ for $\ell\le -3/2$ as follows:

\begin{theorem}\label{thm:3.2}
    Let $\eta,\ell\in \R$ with $\ell \le -3/2$, $\eta\ne 0$. then $\varphi_{\ell,\eta}$ has countably many zeros on $\R$ which are all simple. Moreover, the interlacing property for the real zeros of $\varphi_{\ell,\eta}$ and $\varphi_{\ell+1,\eta}$ holds according to the following pattern\,{\rm :}
    \begin{equation*}
        0< \rho_{\ell+1,\eta,1} < \rho_{\ell,\eta,1} < \rho_{\ell+1,\eta,2} < \rho_{\ell,\eta,2} < \cdots.
    \end{equation*}
\end{theorem}

\begin{proof}
    We proceed to apply Sturm's oscillation theorem (see \cite[p. 53]{Al-Gwaiz2008}) on the differential equation \eqref{ode}, which assures that $F_{\ell,\eta}$ (equivalently $\varphi_{\ell,\eta}$) has countably many zeros on $\R$, regardless of the choice of $\eta,\ell$. Recall that those zeros are denoted by $\big\{ x_{\ell,\eta,k} \big\}_{k=1}^\infty$. Let us assume $\ell<-3/2$, $\ell \ne -\mathbb{N}/2$, and we define 
    \begin{equation*}
        h_{\ell,\eta}(x) = (\ell+1)x^{-2(\ell+1)}W\qb{ U_\ell,\, V_{\ell+1} }(x).
    \end{equation*}
    Then Lemma \ref{lem:3.1} yields that for $\xi_1 \in \big\{ x_{\ell,\eta,k} \big\}_{k=1}^\infty$ and $\xi_2 \in \big\{ x_{\ell+1,\eta,k} \big\}_{k=1}^\infty$, both $h_{\ell,\eta}(\xi_1)>0$ and $h_{\ell,\eta}(\xi_2)>0$ hold. In other words, it can be respectively expressed as
    \begin{equation}\label{sign1}
        -(\ell+1)\varphi_{\ell,\eta}'(\xi_1) \varphi_{\ell+1,\eta}(\xi_1)>0,\quad (\ell+1)\varphi_{\ell,\eta}(\xi_2) \varphi_{\ell+1,\eta}'(\xi_2)>0,
    \end{equation}
    leading that the zeros of $\varphi_{\ell,\eta}$ and $\varphi_{\ell+1,\eta}$ are all simple, since $\varphi_{\ell,\eta}'(\xi_1)\ne 0$ and $\varphi_{\ell+1,\eta}'(\xi_2)\ne 0$. 
    
    In view of \eqref{reflect1}, it suffices to verify the interlacing property on $(0,\infty)$. Let $\rho,\, \bar{\rho}\in \big\{ \rho_{\ell,\eta,k}\big\}_{k=1}^\infty$ be any consecutive positive zeros of $\varphi_{\ell,\eta}$ such that $\rho<\bar{\rho}$. Then it is straightforward from the simplicity of zeros of $\varphi_{\ell,\eta}$ that $\varphi_{\ell,\eta}'(\rho)\varphi_{\ell,\eta}'(\bar{\rho})<0$, and hence, by \eqref{sign1}, we find that $\varphi_{\ell+1,\eta}(\rho)\varphi_{\ell+1,\eta}(\bar{\rho})<0$, which implies that $\varphi_{\ell+1,\eta}$ has odd number of zeros in $\rb{\rho,\bar{\rho}}$. 
    
    Regarding the existence of zero of $\varphi_{\ell+1,\eta}$ within the interval $\big(0,\rho_{\ell,\eta,1}\big)$ where $\rho_{\ell,\eta,1}$ represents the first positive zero of $\varphi_{\ell,\eta}$, we claim that
    $$\varphi_{\ell+1,\eta}(0)\varphi_{\ell+1,\eta}\big( \rho_{\ell,\eta,1} \big)<0.$$
    Since every positive zero of $\varphi_{\ell,\eta}$ is simple, and $\varphi_{\ell,\eta}$ is continuous, it is apparent that $\varphi_{\ell,\eta}(0)\varphi_{\ell,\eta}'\big( \rho_{\ell,\eta,1} \big)<0$. Consequently, by multiplying expressions in \eqref{sign1} and using the relation $\varphi_{\ell+1,\eta}(0)= \varphi_{\ell,\eta}(0)/\qb{2(2\ell+3)(\ell+1)}$, we deduce
    \begin{align*}
        0 &< -(\ell+1)\varphi_{\ell+1,\eta}^2(0)\varphi_{\ell,\eta}'\big( \rho_{\ell,\eta,1} \big) \varphi_{\ell+1,\eta}\big( \rho_{\ell,\eta,1} \big) \\
        &= \frac{-\varphi_{\ell,\eta}(0)\varphi_{\ell,\eta}'\big( \rho_{\ell,\eta,1} \big)}{2(2\ell+3)} \varphi_{\ell+1,\eta}(0)\varphi_{\ell+1,\eta}\big( \rho_{\ell,\eta,1} \big),
    \end{align*}
    and hence the claim is now proved.

    As verified before, $\varphi_{\ell+1,\eta}$ has odd number of positive zeros on any subinterval of $(0,\infty)$ partitioned by $\big\{ \rho_{\ell,\eta,k}\big\}_{k=1}^\infty $. 
    It remains to establish the uniqueness of a zero in $\varphi_{\ell+1,\eta}$ in each of subintervals, say $I$. For the sake of contradiction, we suppose that $\varphi_{\ell+1,\eta}$ has at least two zeros in $I$, denoting $\omega$ and $\bar{\omega}$ as any consecutive zeros of $\varphi_{\ell+1,\eta}$ within $I$. We observe from \eqref{sign1} that
    \begin{equation*}
        (\ell+1)^2\varphi_{\ell,\eta}(\omega) \varphi_{\ell,\eta}(\bar{\omega})\varphi_{\ell+1,\eta}'(\omega) \varphi_{\ell+1,\eta}'(\bar{\omega})>0
    \end{equation*}
    which indicates that $\varphi_{\ell+1,\eta}'\rb{w}$ and $\varphi_{\ell+1,\eta}'\rb{\bar{w}}$ have the same sign as $\varphi_{\ell,\eta}$ maintains constant sign on $I$. Thus it is plain to see that there exists at least one zero of $\varphi_{\ell+1,\eta}$ on $\rb{\omega,\bar{\omega}}$. This contradicts the choice of $\omega$ and $\bar{\omega}$.

    If $\ell \in -\mathbb{N}/2 \cap (-\infty , -3/2]$, the result follows from \eqref{limit1} and Theorem \ref{thm:3.1}, and thus we completes the proof.
\end{proof}

    \begin{figure}[!ht]
    \centering
    \includegraphics[width=0.495\textwidth]{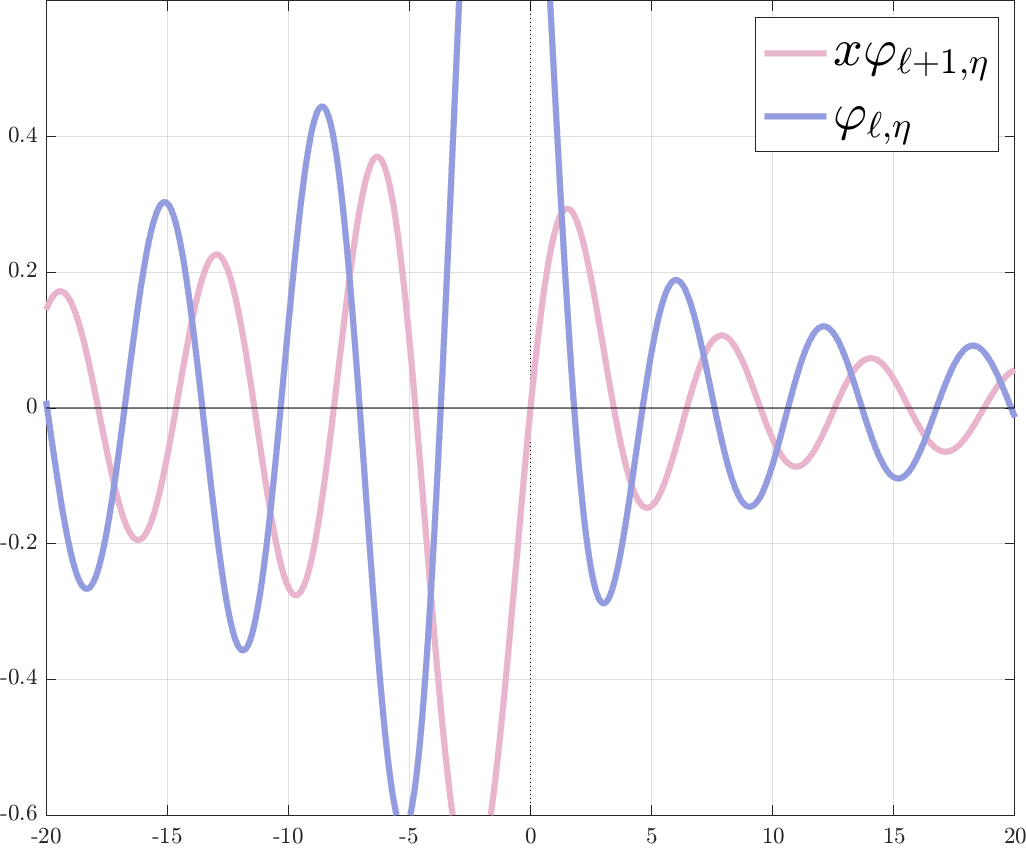}
    \includegraphics[width=0.495\textwidth, height=0.412\textwidth]{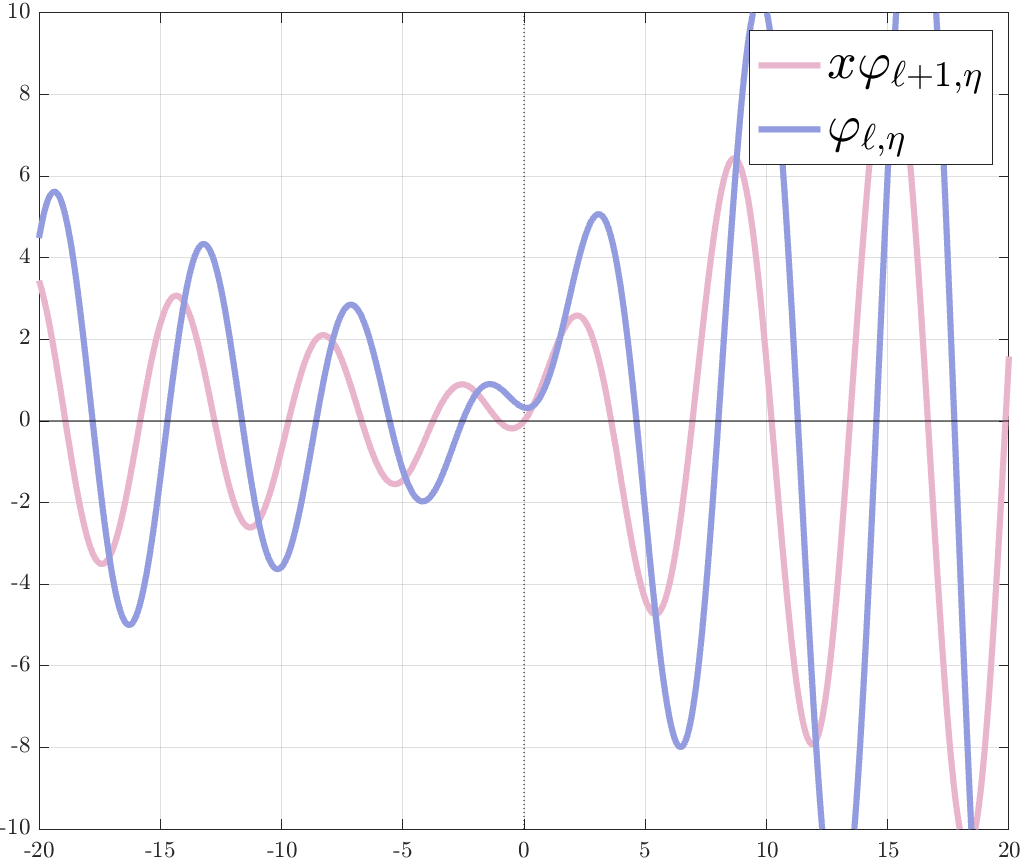}
    \caption{The graphs of $\varphi_{\ell,\eta}$ and $x\varphi_{\ell+1,\eta}$ when $\eta = \ell = -1/3$ (left) and $\eta = 1/3$ and $\ell = -5/3$ (right).}
    \label{fig:3.1}
    \end{figure}

    \begin{remark}\
    \begin{enumerate}[label=(\roman*)]
        \item In the numerical simulations performed in this paper, we employed the \emph{Special Functions in Physics (SpecFunPhys) Toolbox}, as introduced in \cite[Chap. 7]{Schweizer2021}. Since $F_{\ell,\eta}(x)$ gets complex-valued for $\ell <-1$ and $x<0$, $\varphi_{\ell,\eta}$ or $\phi_{\ell,\eta}$ were used for the illustration.

        \item The left and right in Figure \ref{fig:3.1} illustrate zero configuration stated in Theorem \ref{thm:3.1} and \ref{thm:3.2}, respectively.
    \end{enumerate}
    \end{remark}

\section{Monotonicity of the zeros of \texorpdfstring{$F_{\ell,\eta}$}{F ell eta}}

With $\varphi_{\ell,\eta}(x)$ being holomorphic in $\ell\in \rb{-1,\infty}$ and entire in $x$, the implicit function theorem ensures that the function $\ell \mapsto \rho_{\ell,\eta,k}$, where $k\in \mathbb{N}$, is continuously differentiable for $\ell >-1$. Likewise, the function $\eta \mapsto \rho_{\ell,\eta,k}$ is continuously differentiable for $\eta \in \R$. In this section, we investigate the monotonicity of $\rho_{\ell,\eta}$ in terms of $\eta,\ell$ and thereby we answer \cite[Open problem 1-2]{Baricz-C-D-T2016}.

Ismail and Zhang \cite{IsmaiZhangl1988} introduced a novel approach to investigate the monotonicity of eigenvalues in a self adjoint operator, namely the Hellmann-Feynman theorem, which is well-known in quantum chemistry. We refer the readers to \cite{Ismail1987,IsmailMuldoon1988} for the finite-dimensional version and the zeros of Bessel functions (when $\eta=0$) of the Hellmann-Feynman theorem, respectively. In this section, we shall apply this argument, particularly as used in \cite{Baricz2017-2}, to derive Hellmann-Feynman type theorem for the zeros of $F_{\ell,\eta}$.

\begin{theorem}\label{thm:4.1}
For any fixed $k\in \mathbb{N}$ and $\eta \in \R$, the function $\ell \mapsto \rho_{\ell,\eta,k}$ is increasing on the interval $\rb{-1/2,\infty}$, that is,
\begin{equation*}
\frac{d \rho_{\ell,\eta,k}}{d \ell} > 0. \tag{$\ell>-1/2$}
\end{equation*}
\end{theorem}

\smallskip
\begin{proof}
Scaling the differential equation in \eqref{ode} yields
\begin{equation*}
    \frac{d^2 y}{dx^2} + \rb{ \lambda^2 - \frac{2\eta \lambda}{x} - \frac{\ell(\ell+1)}{x^2} }y =0,
\end{equation*}
where $y(x) = F_{\ell,\eta}(\lambda x)$. Consider the differential operator $H_\ell = -\frac{d^2}{dx^2} + \frac{\ell(\ell+1)}{x^2}$, which is self-adjoint with respect to the canonical inner product in $L^2(0,1)$. Then it satisfies the equation $H_\ell y_\ell = \Lambda_\ell y_\ell$, where $\Lambda_\ell = \rb{\lambda_\ell^2 -\frac{2\eta \lambda_\ell}{x}} y_\ell$. 

We recall that if $\ell>-1/2$ then $F_{\ell,\eta}$ has only real zeros, as shown by Theorem \ref{thm:C}. Let $\lambda_\ell$ denote any positive real zero of $F_{\ell,\eta}$ apart from the origin, \rm{i.e.}, $\lambda_\ell = \rho_{\ell,\eta,k}$ for $k\in \mathbb{N}$. Given that the expression $y_\ell'(x)y_L(x)-y_\ell(x)y_L'(x)$ vanishes at $x=0,\,1$ when $\ell>-1/2$,
it follows that for $-1/2<\ell <L$,
\begin{equation}\label{Inner1}
    \begin{split}
        \ab{ H_\ell y_\ell , y_L } - \ab{ y_\ell, H_L y_L } &= \ab{ \rb{H_\ell - H_L} y_\ell, y_L }\\
        &= \big( \ell(\ell+1)-L(L+1) \big) \int_0^1 y_\ell(t) y_L(t) \frac{dt}{t^2}. 
    \end{split}
\end{equation}
On the other hand, we have
\begin{equation}\label{Inner2}
    \ab{ H_\ell y_\ell , y_L } - \ab{ y_\ell, H_L y_L } = \ab{ \rb{ \Lambda_\ell - \Lambda_L }y_\ell,y_L }.
\end{equation}
Dividing by $\ell-L$ and letting $L\to \ell$ in \eqref{Inner1} and \eqref{Inner2}, we observe that
\begin{equation}\label{similar}
\begin{aligned}
    (2\ell+1) \int_0^1 (y_\ell(t))^2 \frac{dt}{t^2} &= \lim_{L\to \ell} \ab{ \frac{H_\ell - H_L}{\ell-L} y_\ell,y_L }  \\
    &=  \ab{ \frac{d\Lambda_\ell}{d \ell} y_\ell,y_L }  \\
    &= 2\frac{d\lambda_\ell}{d \ell} \rb{ \lambda_\ell \int_0^1 (y_\ell(t))^2 dt - \eta \int_0^1 (y_\ell(t))^2 \frac{dt}{t} }, 
\end{aligned}
\end{equation}
in which integrals converge for $\ell>-1/2$.

We now claim that for fixed $\ell>-1/2$ and $\eta\in \R$,
\begin{equation}\label{HF1}
    \Psi_{\ell,\eta} := \lambda_\ell \int_0^1 (y_\ell(t))^2 dt - \eta \int_0^1 (y_\ell(t))^2 \frac{dt}{t}>0.
\end{equation}
It is obvious that $\Psi_\ell >0$ if $\eta \le 0$.
In the remaining case when $\eta>0$, we consider the identity $\ab{H_\ell y_\ell,y_\ell} = \ab{\Lambda_\ell y_\ell,y_\ell}$, which results in
\begin{align}
    \Psi_{\ell,\eta} &= \frac{1}{\lambda_\ell} \rb{\int_0^1 \rb{y_\ell'(t)}^2 dt + \ell(\ell+1) \int_0^1 \rb{y_\ell(t)}^2 \frac{dt}{t^2} } +\eta \int_0^1 \rb{y_\ell(t)}^2\frac{dt}{t} \notag\\
    & >0 \ifz \eta >0, \label{psi1}
\end{align}
for fixed $\ell\ge 0$. Moreover, by Hardy's inequality (see \cite[p. 243]{Hardy-L-P1952}) with $p=2$, we have, for $-1/2<\ell < 0$,
\begin{equation*}
    \int_0^1 \rb{y_\ell'(t)}^2 dt > \rb{\frac{2-1}{2}}^2 \int_0^1 \rb{y_\ell(t)}^2 \frac{dt}{t^2} \ge -\ell(\ell+1) \int_0^1 \rb{y_\ell(t)}^2 \frac{dt}{t^2}.
\end{equation*}
Thus \eqref{psi1} remains true for fixed $-1/2<\ell < 0$.

Therefore, the desired result follows from the expression
\begin{equation*}
    \frac{d\lambda_\ell}{d \ell} = \frac{\ell+1/2}{\Psi_{\ell,\eta}} \int_0^1 (y_\ell(t))^2 \frac{dt}{t^2}>0.
\end{equation*}
\end{proof}

In light of \eqref{reflect1} and the above theorem, it is evident that the negative zeros of $F_{\ell,\eta}$ are decreasing on $(-1/2,\infty)$ with respect to $\ell$ for given $\eta$.
As indicated in Figure \ref{fig:4.1}, we can see that the positive (resp. nagative) zeros are increasing (resp. decreasing) on $(-1/2,\infty)$ with respect to $\ell$, while the zeros reveal more complicated pattern on $(-\infty,-1/2)$. We also note that the trajectories when $\eta=1/5$ will be the exact reflection of Figure \ref{fig:4.1} with respect to the $\ell$-axis.
The zero-variation map corresponding to the case where $\eta =0$ can be found in \cite[\S 15.6]{Watson1922} and \cite[p. 9]{ChoChung2021}.

\begin{figure}[!ht]
    \centering
    \includegraphics[width=0.495\textwidth, height=0.412\textwidth]{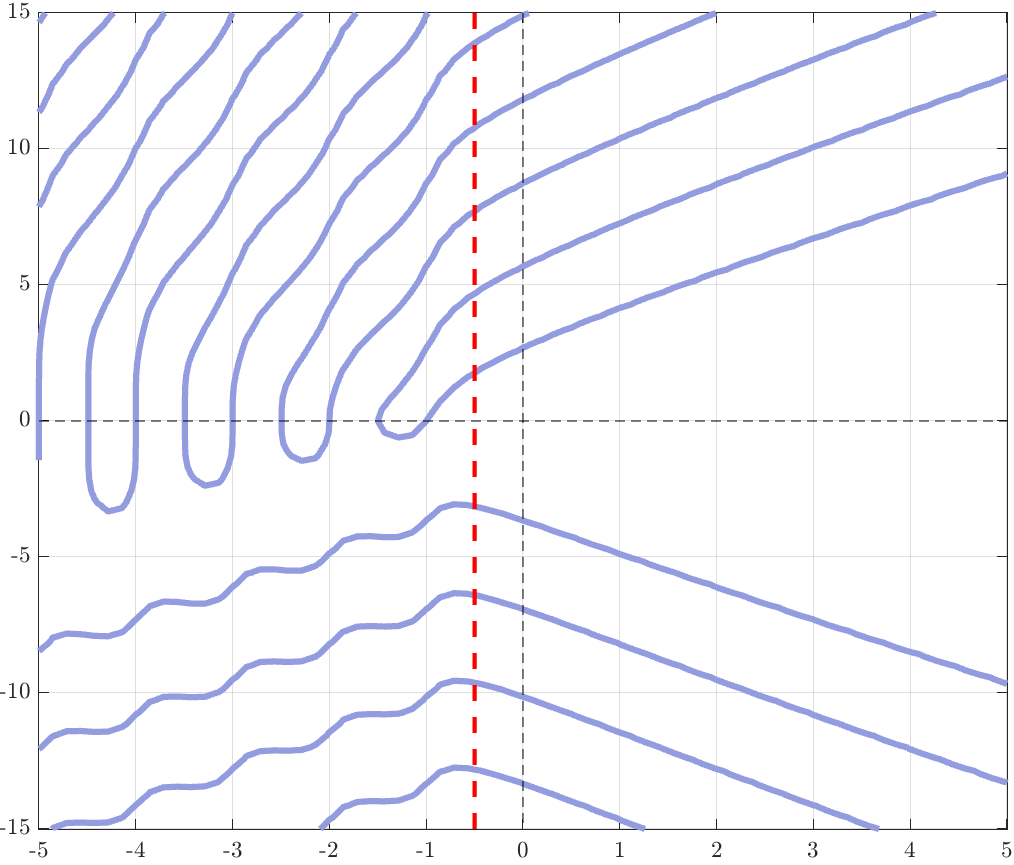}
    \caption{The trajectories of the real zeros of $F_{\ell,\eta}$ in the $\rb{\ell,x}$-plane are depicted for $-5\le \ell \le 5$, $-15\le x\le 15$ and $\eta=-1/5$. The red dotted line represents vertical line $\nu=-1/2$.} 
    \label{fig:4.1}
\end{figure}

An immediate consequence of Theorems \ref{thm:3.1} and \ref{thm:4.1} gives the following interlacing pattern for $-1/2<\ell < L <\ell+1$:
\begin{equation*}
    0< \rho_{\ell,\eta,1} < \rho_{L,\eta,1} < \rho_{\ell+1,\eta,1} < \rho_{\ell,\eta,2} < \rho_{L,\eta,2} < \rho_{\ell+1,\eta,2} < \cdots
\end{equation*}
which can simply stated as follows:
\begin{corollary}
    Let $-1/2<\ell<L$ with $\vb{\ell-L}\le1$, and let $\eta\in \R$. Then the zeros of $\varphi_{\ell,\eta}$, $\varphi_{L,\eta}$ are interlaced according to the following pattern{\rm :}
    \begin{equation*}
        0< \rho_{\ell,\eta,1} < \rho_{L,\eta,1} < \rho_{\ell,\eta,2} < \rho_{L,\eta,2} < \cdots.
    \end{equation*}
\end{corollary}

Similarly, for a fixed $\ell > -1/2$, we can deduce the monotonicity of the zero $\rho_{\ell,\eta,k}$ with respect to $\eta$ as follows:

\begin{theorem}\label{thm:4.2}
For any fixed $k \in \mathbb{N}$ and $\ell >-1/2$, the function $\eta \mapsto \rho_{\ell,\eta,k}$ are increasing on $\R$, that is,
\begin{equation*}
\frac{d \rho_{\ell,\eta,k}}{d \eta} > 0. \tag{$\eta\in \R$}
\end{equation*}
\end{theorem}

\smallskip
\begin{proof}
Let $\ell>-1/2$ be fixed. In an analogous argument to the proof of Theorem \ref{thm:4.1}, if we replace the parameter $\ell$ with $\eta$ and consider $\lambda_\eta = \rho_{\ell,\eta,k}$ where $k\in \mathbb{N}$, a similar result of \eqref{similar} in terms of $\eta$ can be derived as
\begin{equation*}
    \frac{d\lambda_\eta}{d \eta} = \frac{\lambda_\eta}{\Psi_{\ell,\eta}^*} \int_0^1 (y_\eta(t))^2 \frac{dt}{t},
\end{equation*}
where $y_\eta(x) = F_{\ell,\eta}(\lambda_\eta \,x)$ and
\begin{equation*}
    \Psi_{\ell,\eta}^* = \lambda_\eta \int_0^1 (y_\eta(t))^2 dt - \eta \int_0^1 (y_\eta(t))^2 \frac{dt}{t},
\end{equation*}
Moreover, the same reasoning used in \eqref{HF1} shows that $\Psi_{\ell,\eta}^* >0$ for $\ell>-1/2$ and $\eta\in \R$, leading to the desired result.
\end{proof}

Figure \ref{fig:3.2} depicts that $\eta \mapsto x_{\ell,\eta,k}$, $k\in \mathbb{N}$ is an increasing function on $\R$ for fixed $\ell>-1/2$. Nevertheless, numerical simulations suggest that the first positive and negative zeros may not consistently increase over the entire real line when $\ell \in (-n-1/2,-n)$ for $n\in \mathbb{N}$.

\begin{figure}[!ht]
    \centering
    \includegraphics[width=0.495\textwidth, height=0.412\textwidth]{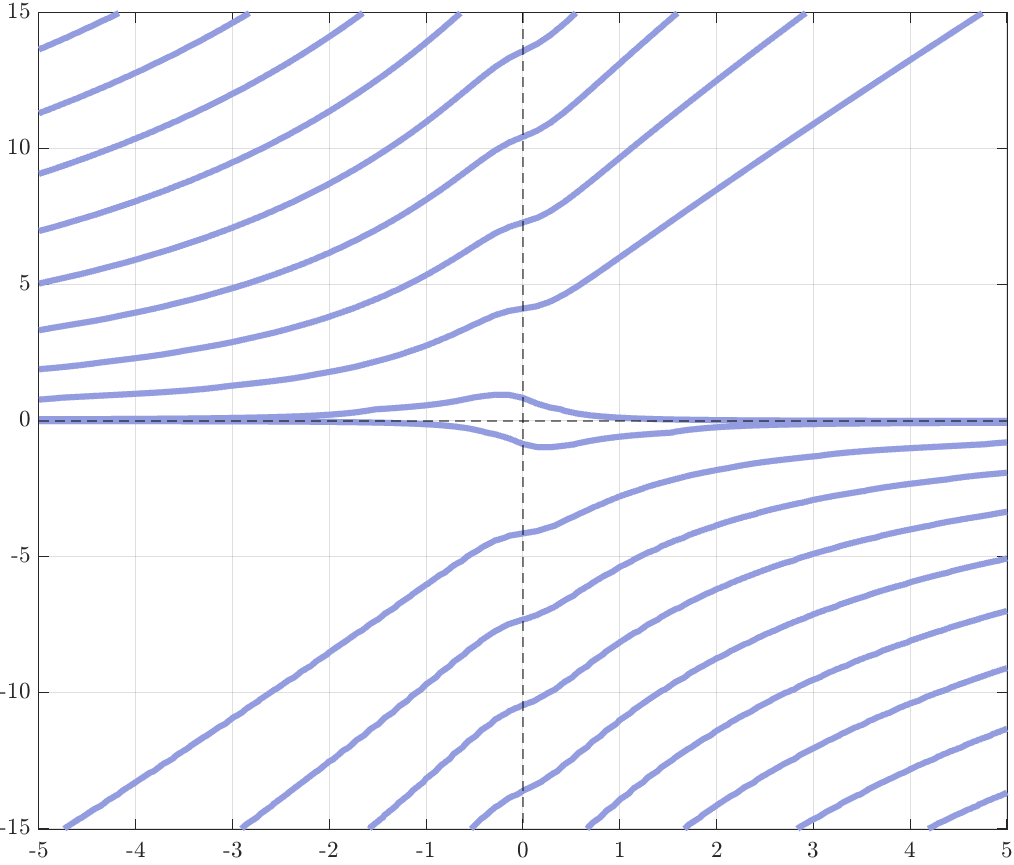}
    \includegraphics[width=0.495\textwidth, height=0.412\textwidth]{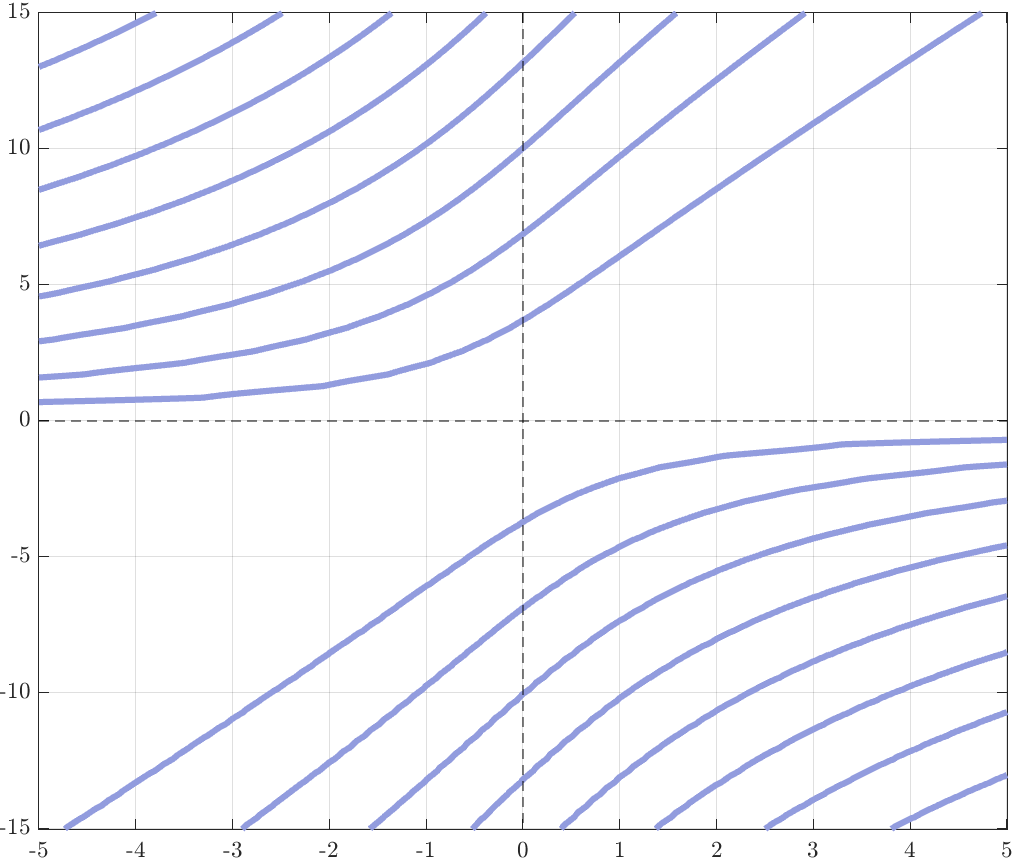}
    \caption{The trajectories of the zeros of $F_{\ell,\eta}$ in the $\rb{\eta,x}$-plane are depicted for $-5\le \eta \le 5$ and $-15\le x\le 15$.  The left plot corresponds to $\ell=-4/3$, while the right one corresponds to $\ell=1/5$.}
    \label{fig:3.2}
\end{figure}

\section{Orthogonal polynomials in Pad{\'e} approximants}

Wimp \cite{Wimp1985} introduced a sequence of orthogonal polynomials to serve as numerator and denominator polynomials in the Pad{\'e} approximant for the logarithmic derivative of the confluent hypergeometric function, specifically ${}_1F_1\rb{c/2+1 -i\kappa, c+2; i/x}$. To provide more clarity, a key outcome in the aforementioned work, corresponding to $c= 2\ell +2$ and $\kappa =\eta$, is as follows:
\begin{equation*}
    \lim_{n\to \infty} \frac{R_{n-1,\ell+1,\eta}(x)}{R_{n,\ell,\eta}(x)}   = \frac{\phi_{\ell+1,\eta}\rb{\frac{1}{2x}}}{x\phi_{\ell,\eta}\rb{\frac{1}{2x}}} = \int_{\text{supp}(\mu)} \frac{d\mu}{x-t},
\end{equation*}
where $\cb{R_{n,\ell,\eta}(x)}$ denotes a monic sequence of orthogonal polynomials for $\ell\in (-3/2,\infty)\setminus \cb{-1}$ and $\eta\ne0$, with respect to purely discrete measure $d\mu$, arising from the three-term recurrence relation
\begin{align}\label{poly1}
    y_{n+1}&=\rb{x-\alpha_n}y_{n}-\beta_n y_{n-1},
\end{align}
with
\begin{equation}\label{beta}
    \alpha_n = -\frac{\eta}{2(\ell+n+1)(\ell+n+2)},\quad
    \beta_n =  \frac{(\ell+n+1)^2 + \eta^2}{2(2\ell+2n+1)_3(\ell+n+1)}.
\end{equation}
The initial values of the above are given by
\begin{equation}\label{list}
    R_{0,\ell,\eta}(x) = 1,\quad R_{1,\ell,\eta}(x) = x- \alpha_{0}= x + \frac{\eta}{2(\ell+1)(\ell+2)}.
\end{equation}
Particularly noteworthy in \cite{Wimp1985} is the explicit expression of the form
\begin{equation}\label{explicit}
\begin{multlined}[c][.9\displaywidth]
    R_{n,\ell,\eta}(x) = \sum_{k=0}^n \frac{(-\ell-i\eta-n-1)_k i^{k} x^{n-k}}{k! (-2\ell-2n-2)_k} \\
    \cdot {}_3F_2\left[\begin{array}{c} -k,\, 2n+2\ell+3-k,\, i\eta+\ell+1\\
n+i\eta+\ell+2-k,\, 2\ell+2 \end{array}\biggr| \,1\right],
\end{multlined}
\end{equation}
which is useful in analyzing how those polynomials are related to the regular Coulomb wave function $F_{\ell,\eta}$.

Unaware of the Wimp's work, {\v S}tampach and {\v S}{\v t}ov{\' i}{\v c}ek \cite{StampachStovicek2014} also introduced an equivalent sequence of orthogonal polynomials in their study of the Jacobi operator.  Following the expressions in \cite{StampachStovicek2014}, the associated orthogonal polynomials satisfy the three term recurrence relation
\begin{equation*}
    x P_n(x) = w_{n-1}P_{n-1}(x) + \lambda_n P_n(x) + w_n P_{n+1}(x), \for n\in \mathbb{N},
\end{equation*}
which is equivalent to \eqref{poly1} under the choice of $\lambda_n = 2\alpha_n$ and $w_n = 2\sqrt{\beta_{n+1}}$ (see \cite[(2)]{StampachStovicek2014}). The solutions to the above recurrence relation with initial conditions $P_{-1}(x) =0 $ and $P_0(x)=1$ can be explicitly written as
\begin{equation}\label{PR1}
    P_n(x) = \frac{1}{\sqrt{\zeta_n}} R_{n,\ell,\eta}(x/2),\quad
    \zeta_n = \prod_{k=1}^n \beta_k >0.
\end{equation}
It is also important to mention that the Proposition 8 in \cite[(56-57)]{StampachStovicek2014} states that for $x\ne 0$,
\begin{equation}\label{limit3}
    \lim_{n\to \infty} (2x)^{n} R_{n,\ell,\eta}\rb{\frac{1}{2x}} = \phi_{\ell,\eta}(x).
\end{equation}
Moreover, by taking $\eta=0$ and $\ell = \nu-1/2$, it reduces to the Lommel polynomials (see \cite[\S 9.6]{Watson1922}) with the relation
\begin{equation*}
    2^{2n}(\nu+1)_n R_{n,\ell,\eta}\rb{\frac{1}{2x}} = R_{n,\nu+1}(x).
\end{equation*}

\begin{remark}\label{rem:5.1}
In fact, the limit \eqref{limit3} has uniform convergence on any compact subsets of $\mathbb{C}$. To justify,
we observe from \eqref{explicit} that
\begin{equation*}
    (2z)^{n} R_{n,\ell,\eta}\rb{\frac{1}{2x}} = \sum_{k=0}^n \sum_{m=0}^k A_{k,m}(n) B_{k,m} x^k,
\end{equation*}
where
\begin{equation*}
    A_{k,m}(n) = \frac{(-\ell-i\eta-n-1)_{k-m} }{(-2\ell-2n-2)_{k-m}},\quad B_{k,m} =  \frac{(-k)_m(i\eta + \ell +1)_m}{k!\, m! (2\ell +2)_m} (2i)^k.
\end{equation*}
Since $ \lim_{n\to \infty} A_{k,m}(n) = 2^{m-k}$, we find that $|A_{k,m}(n)| \le 1$ for $0\le m \le k$ and sufficiently large $n$.
On the other hand, one may easily verify that $2^{m-k} B_{k,m}$ satisfies the recurrence \eqref{series2} so that using \eqref{series1},
\begin{equation}\label{phi2}
    \phi_{\ell,\eta}(x) = \sum_{k=0}^\infty a_{\ell,\eta,k} x^{k} =\sum_{k=0}^\infty \sum_{m=0}^k 2^{m-k} B_{k,m} x^k.
\end{equation}
It is simple to see that $\limsup_{k\to \infty} \sum_{m=0}^{k+1} |B_{k+1,m}|\big/ \sum_{m=0}^{k} |B_{k,m}| =0 $ by using stirling's formula. Consequently, the series in \eqref{phi2} converges absolutely and uniformly on any compact subsets of $\mathbb{C}$.
Hence the result follows from the dominated convergence theorem (specifically, Tannery theorem).
\end{remark}

In order to further elaborate results on Theorem \ref{thm:B} and \cite[Open problem 3]{Baricz-C-D-T2016}, we may consider the polynomials associated with Dini-like function $x F_{\ell,\eta}'(x) +h F_{\ell,\eta}(x)$, $h\in \R$.
\begin{definition}
    Let $\eta,\ell\in \R$ with $\ell >-3/2$, $\ell\ne-1$ and $\eta\ne 0$, and let $H\in \R$. We define
    \begin{equation*}
    D_{n,\ell,\eta}(H; x) =\rb{\frac{\eta}{2(\ell+1)}+ H x} R_{n,\ell,\eta}(x) - \frac{1+\eta^2/(\ell+1)^2}{4(2\ell+3)} R_{n-1,\ell+1,\eta}(x).
    \end{equation*}
\end{definition}

Regarding \eqref{explicit}, it is clear that $D_{n,\ell,\eta}(H;x)$ is a polynomial of degree $(n+1)$ under the same conditions for $\eta,\ell$.

\begin{proposition}\label{prop:5.1}
Let $\eta,\ell\in \R$ with $\ell >-3/2$, $\ell\ne-1$ and $\eta\ne 0$. Then
\begin{equation*}
    \lim_{n\to\infty} (2x)^{n+1} D_{n,\ell,\eta}\rb{H; \frac{1}{2x}} = x\phi'_{\ell,\eta}(x) + H\phi_{\ell,\eta}(x)
\end{equation*}
for each $x\in \R$. Moreover, the convergence is uniform on any compact subset of $\R$.
\end{proposition}

\begin{proof}
We begin by reformulating \eqref{recur1} as, for $\ell >-3/2$ and $\ell\ne-1$,
\begin{equation*}
    \phi_{\ell,\eta}'(x) = \frac{\eta}{\ell+1}\phi_{\ell,\eta}(x) - \frac{x}{2\ell+3}\rb{ 1+\frac{\eta^2}{(\ell+1)^2} }\phi_{\ell+1,\eta}(x).
\end{equation*}
A simple consequence of \eqref{limit3} shows that
\begin{align*}
    \lim_{n\to\infty} &(2x)^{n+1} D_{n,\ell,\eta}\rb{H; \frac{1}{2x}} \\
    &= \rb{\frac{\eta}{\ell+1}x + H}\phi_{\ell,\eta}(x) -\frac{1+\eta^2/(\ell+1)^2}{2\ell+3} x^2 \phi_{\ell+1,\eta}(x)\\
    &= x\phi'_{\ell,\eta}(x) + H\phi_{\ell,\eta}(x)
\end{align*}
for each $x\ne 0$. In view of expressions \eqref{explicit} and \eqref{phi}, $x^n R_{n,\ell,\eta}(1/(2x))$ is well-defined at $x=0$, and moreover, it follows that
\begin{equation*}
    (2x)^{n+1} D_{n,\ell,\eta}\rb{H; \frac{1}{2x}} = H = x\phi'_{\ell,\eta}(x) + H\phi_{\ell,\eta}(x)
\end{equation*}
at $x=0$. Since the convergence follows directly from Remark \ref{rem:5.1}, the proof is complete.
\end{proof}

\begin{theorem}\label{thm:5.1}
Let $\eta,\ell\in \R$ with $\ell >-3/2$, $\ell\ne-1$ and $\eta\ne 0$, and let $n\in \mathbb{N}$. If $H\ge 0$, the zeros of polynomials $R_{n,\ell,\eta}(x)$ and $D_{n,\ell,\eta}(H;x)$ are all simple and real, and those zeros are interlaced each other.
\end{theorem}

\begin{proof}
    Owing to \cite[Theorem 2.2.3]{Ismail2005}, we find that all zeros of $R_{n,\ell,\eta}(x)$ are real and simple for given $\ell >-3/2$, $\ell\ne -1$, and $\eta \ne 0$. We observe that
    \begin{equation}\label{mero1}
        \frac{d}{dx}\frac{D_{n,\ell,\eta}(H; x)}{ R_{n,\ell,\eta}(x)} = H + \frac{1+\eta^2/(\ell+1)^2}{4(2\ell+3)} \frac{W[R_{n-1,\ell+1,\eta},R_{n,\ell,\eta}](x)}{R_{n,\ell,\eta}^2(x)}
    \end{equation}
    Using \eqref{PR1}, we rephrase the formula \cite[p. 248]{StampachStovicek2014} as
    \begin{equation*}
        \frac{R_{m,\ell-1,\eta}(x)R_{m+s,\ell,\eta}(x)}{\sqrt{\zeta_{m}}\sqrt{\zeta_{m+s}}} - \frac{R_{m+s+1,\ell-1,\eta}(x)R_{m-1,\ell,\eta}(x)}{\sqrt{\zeta_{m+s+1}}\sqrt{\zeta_{m-1}}} =  \frac{R_{s,\ell+m,\eta}(x)}{\sqrt{\beta_m}} 
    \end{equation*}
    for $m,s\in \mathbb{Z}_+$, where $\beta_m>0$ and $\zeta_m>0$ are presented in \eqref{beta} and \eqref{PR1} respectively. If we set $s=n-1$ and $m=1$, we have
    \begin{equation*}
        \frac{R_{1,\ell-1,\eta}(x)R_{n,\ell,\eta}(x)}{\sqrt{\zeta_{1}}\sqrt{\zeta_{n}}} - \frac{R_{n+1,\ell-1,\eta}(x)R_{0,\ell,\eta}(x)}{\sqrt{\zeta_{n+1}}\sqrt{\zeta_{0}}} =  \frac{R_{n-1,\ell+1,\eta}(x)}{\sqrt{\beta_1}} 
    \end{equation*}
    \begin{equation*}
        \frac{1}{\sqrt{\zeta_{n}}}R_{1,\ell-1,\eta}(x)R_{n,\ell,\eta}(x) - \frac{\sqrt{\beta_1}}{\sqrt{\zeta_{n+1}}}R_{n+1,\ell-1,\eta}(x) =  R_{n-1,\ell+1,\eta}(x)
    \end{equation*}
    If we differentiate both sides after dividing by $\sqrt{\zeta_{n+1}} R_{n,\ell,\eta}(x)/ \sqrt{\beta_1}$, it is simple to deduce
    \begin{equation*}
        W [ R_{n,\ell,\eta}, R_{n+1,\ell-1,\eta} ](x) = \sqrt{ \frac{\zeta_{n+1}}{\beta_1} } W [ R_{n-1,\ell+1,\eta}, R_{n,\ell,\eta} ](x) + \sqrt{ \frac{\beta_{n+1}}{\beta_1} } R_{n,\ell,\eta}^2(x),
    \end{equation*}
    which implies that
    \begin{equation}\label{mero2}
        W [ R_{n,\ell,\eta}, R_{n+1,\ell-1,\eta} ](x) = \sum_{k=0}^n \rb{\frac{ \beta_{n+1-k}\prod_{j=1}^k \zeta_{n+1-j}  }{ \beta_1^{k+1}}}^{1/2} R_{n-k,\ell,\eta}^2(x)>0.
    \end{equation}
    Hence if $H \ge 0$, then the rational function $D_{n,\ell,\eta}(H,x)/R_{n,\ell,\eta}(x)$ is strictly increasing on each of subintervals of $\R$, partitioned by the zeros of $R_{n,\ell,\eta}(x)$. Consequently $D_{n,\ell,\eta}(H,x)$ has one and only one zero between two consecutive zeros of $R_{n,\ell,\eta}(x)$. Moreover, since
    \begin{equation*}
        \frac{D_{n,\ell,\eta}(H; x)}{ R_{n,\ell,\eta}(x)} = 2H x + O(1) \quad\text{as }|x| \to \infty,
    \end{equation*}
    $D_{n,\ell,\eta}(H,x)$ has one and only one zero on each of intervals $(-\infty,r_{\min})$ and $(r_{\max},\infty)$, where $r_{\min}$ and $r_{\max}$ denote respectively the smallest and largest zeros of $R_{n,\ell,\eta}(x)$. Therefore, $D_{n,\ell,\eta}(H,x)$ has $(n+1)$ zeros on the real line, and those zeros are interlaced with the zeros of $R_{n,\ell,\eta}(x)$, which means that it has only simple real zeros since $D_{n,\ell,\eta}(H,x)$ is a polynomial of degree $(n+1)$.
\end{proof}

In a conventional manner, by applying another Hurwitz's theorem (see for instance \cite[p. 152]{Conway1978}) along with Proposition \ref{prop:5.1}, we conclude that the function $x\phi'_{\ell,\eta}(x) + H \phi_{\ell,\eta}(x)$ has only real zeros if $H\ge 0$. In particular, it is also a simple application of Hadamard expansion \eqref{Hadamard}, which can be stated as:

\begin{theorem}
    Let $\eta,\ell\in \R$ with $\ell >-3/2$, $\ell\ne-1$ and $\eta\ne 0$.
    If $h \ge -\ell-1$, $xF_{\ell,\eta}'(x)+h F_{\ell,\eta}(x)$ has only real zeros. 
\end{theorem}

\begin{proof}
We observe that by taking logarithmic derivative on \eqref{Hadamard}, 
\begin{equation}\label{ML2}
    \frac{\phi_{\ell,\eta}'(x)}{\phi_{\ell,\eta}(x)} = \frac{\eta}{\ell+1} + \sum_{k=1}^\infty \frac{x}{x_{\ell,\eta,k} (x- x_{\ell,\eta,k})}.
\end{equation}
Then we write
\begin{equation*}
    x\phi_{\ell,\eta}'(x) + H\phi_{\ell,\eta}(x) = x\phi_{\ell,\eta}(x)\rb{ \frac{H}{x}+ \frac{\eta}{\ell+1} + \sum_{k=1}^\infty \frac{x}{x_{\ell,\eta,k} (x- x_{\ell,\eta,k})}}.
\end{equation*}
Since $\phi_{\ell,\eta}$ and $\phi_{\ell,\eta}'$ do not share zeros in common (If so, we deduce from \eqref{ode2} that $\phi_{\ell,\eta}^{(n)}$ vanishes for all $n\ge 2$, which is a contradiction), it follows that for any zero $\alpha\in \mathbb{C}\setminus\cb{0}$ of $x\phi_{\ell,\eta}'(x) + H\phi_{\ell,\eta}(x)$, it satisfies
\begin{equation*}
    \frac{H}{\alpha}+ \frac{\eta}{\ell+1} + \sum_{k=1}^\infty \frac{\alpha}{x_{\ell,\eta,k} (\alpha- x_{\ell,\eta,k})} =0.
\end{equation*}
Thus we obtain
\begin{equation*}
    -\rb{\frac{H}{|\alpha|^2}+ \sum_{k=1}^\infty \frac{1}{|\alpha- x_{\ell,\eta,k}|^2} }\text{Im}(\alpha) =0,
\end{equation*}
which implies that $\text{Im}(\alpha) =0$ if $H \ge 0$. By considering
\begin{equation}\label{relation}
    x F_{\ell,\eta}'(x) + h F_{\ell,\eta}(x) = C_{\ell,\eta} x^{\ell+1} \rb{ x\phi_{\ell,\eta}'(x) + (h+\ell+1)\phi_{\ell,\eta}(x) },
\end{equation}
the zeros of $x F_{\ell,\eta}'(x) + h F_{\ell,\eta}(x)$ are all real if $h + \ell+1 \ge 0$.

\end{proof}

Based on the results discussed earlier, we extend Theorem \ref{thm:B} as follows:
\begin{theorem}\label{thm:5.3}
    Let $\eta,\ell\in \R$ with $\ell >-3/2$, $\ell\ne-1$ and $\eta\ne 0$. Then the following hold true:
    \begin{enumerate}[label={\rm(\roman*)}]
        \item The zeros of $x\phi_{\ell,\eta}(x)$ and $x\phi_{\ell,\eta}'(x) + H\phi_{\ell,\eta}(x)$ are interlaced, if $H > 0$.
        \item The zeros of $\phi_{\ell,\eta}(x)$ and $\phi_{\ell,\eta}'(x) $ are interlaced.
    \end{enumerate}
\end{theorem}

\begin{proof}
    Let us consider the meromorphic function
    \begin{equation*}
        Q(x) = \frac{x\phi'_{\ell,\eta}(x) + H \phi_{\ell,\eta}(x) }{ x\phi_{\ell,\eta}(x) }.
    \end{equation*}
    On taking advantage of the limits \eqref{limit3}, Proposition \ref{prop:5.1} and Remark \ref{rem:5.1}, we find that for each $x\in \R \setminus \big\{ x_{\ell,\eta,k}\big\}_{k=1}^\infty \cup \cb{0}$,
    \begin{equation*}
        \lim_{n\to\infty }\frac{ D_{n,\ell,\eta}\rb{H; \frac{1}{2x} }}{ R_{n,\ell,\eta}\rb{ \frac{1}{2x}}} = Q\rb{x}
    \end{equation*}
    As readily verified, by using \eqref{mero1} and \eqref{mero2}, the function $Q(x)$ is decreasing on each subintervals of $\R$ partitioned by its poles $\big\{ x_{\ell,\eta,k}\big\}_{k=1}^\infty \cup \cb{0}$. In this regard, the zeros of $Q(x)$ are interlaced with its poles, provided that $H \ge0$. In particular case of $H=0$, The zeros of $\phi_{\ell,\eta}(x)$ and $\phi_{\ell,\eta}'(x) $ are interlaced to each other since $x$ will be canceled in both numerator and denominator.
\end{proof}

As for the corresponding statement involving $F_{\ell,\eta}(x)$ and $xF_{\ell,\eta}'(x) + h F_{\ell,\eta}(x)$, it can be rewritten by using \eqref{relation}.
Additionally, an analogue of Lemma \ref{lem:2.2} for $F_{\ell,\eta}$ follows immediately from the above theorem.
\begin{corollary}
 Let $\eta,\ell\in \R$ with $\ell >-1$ and $\eta\ne 0$.
The positive zeros of $F_{\ell,\eta}$ are bounded below by $\eta + \sqrt{\eta^2+(\ell+1)^2}>0$.
\end{corollary}

\begin{proof}
Since $x\phi_{\ell,\eta}(x)$ has a zero at the origin but $x\phi_{\ell,\eta}'(x) +(\ell+1)\phi_{\ell,\eta}(x)$ is not, Theorem \ref{thm:5.3}, (ii) with \eqref{relation} implies that the smallest positive zero of $F_{\ell,\eta}'(x)$ is smaller than that of $F_{\ell,\eta}(x)$. Hence the conclusion is immediate from Lemma \ref{lem:2.2}.
\end{proof}

\section{Breaking down of separation theorem}

Unlike the scenario of $\phi_{\ell,\eta}(x)$ and $x \phi_{\ell+1,\eta}(x)$ in section 3,
the zero separation property between $\phi_{\ell,\eta}(x)$ and $x\phi_{\ell+2,\eta}(x)$ is no longer available when $\eta \ne 0$ and $\ell >-1$.
In order to examine more general property between $\phi_{\ell,\eta}(x)$ and $\phi_{\ell+2,\eta}(x)$, we first observe the common zeros of $\phi_{\ell,\eta}(x)$ and $ \phi_{\ell+2,\eta}(x)$. 

\begin{proposition}\label{prop:6.1}
Let $\eta \ne 0$ and $\ell > -1$. Then $\phi_{\ell,\eta}(x)$ and $\phi_{\ell+2,\eta}(x)$ can have at most one zero in common, occurring only at $x=-(\ell+1)(\ell+2)/\eta$ if it exists.
\end{proposition}

\begin{proof}
For $\eta\in \R$ and $\ell>-1$, we begin with writing \eqref{recur2} as
\begin{equation}\label{thm:recur}
\begin{multlined}[c][.9\displaywidth]
    \phi_{\ell,\eta}(x) - 2R_{1,\ell,\eta}\rb{\frac{1}{2x}} x \phi_{\ell+1,\eta}(x)\\
    + \frac{(\ell+2)^2+\eta^2}{(\ell+2)^2(2\ell+3)(2\ell+5)} x^2 \phi_{\ell+2,\eta}(x) =0,
\end{multlined}
\end{equation}
where $R_{1,\ell,\eta}(x)= x + \eta/[2(\ell+1)(\ell+2)]$, as presented in \eqref{list}. Let $\rho^*$ be the common zero of $\phi_{\ell,\eta}(x)$ and $\phi_{\ell+2,\eta}(x)$. In light of Theorem \ref{thm:2.1}, we find that $\phi_{\ell+1,\eta}(\rho^*)\ne 0$. Accordingly, by substituting $\rho^*$ into \eqref{thm:recur}, it is evident that $\rho^*$ is necessarily the zero of $R_{1,\ell,\eta}(1/(2x))$, that is, $\rho^* = -(\ell+1)(\ell+2)/\eta$. 
\end{proof}

\begin{remark}
Practically, the functions $\phi_{\ell,\eta}(x)$ and $\phi_{\ell+2,\eta}(x)$ could have the zero in common for appropriate choices of $\eta,\ell$.
In specific, the numerical simulation indicates that for $\eta = 1/3$, there exists $\ell^* \in (-0.102,-0.103)$ such that $\phi_{\ell^*,1/3}(\rho^*) = \phi_{\ell^* +2 ,1/3}(\rho^*)=0$, where $\rho^* = -(\ell+1)(\ell+2)/\eta$.
\end{remark}

We now establish the generalized interlacing property by supplementing the zero $\rho^*$ of $R_{1,\ell,\eta}$ to the set of zeros of $\phi_{\ell+2,\eta}$.

\begin{theorem}
Let $\eta \ne 0$ and $\ell > -1$, and define $\rho^* = -(\ell+1)(\ell+2)/\eta$.
\begin{enumerate}[label={\rm(\roman*)}]
    \item if $\rho^*\ne x_{\ell,\eta,k}$ for all $k\ge1$, then the zeros of $\phi_{\ell,\eta}(x)$ are interlaced with the zeros of $x (x-\rho^*) \phi_{\ell+2,\eta}(x)$.
    \item if $\rho^*= x_{\ell,\eta,k}$ for some $k\ge1$, then the zeros of $\phi_{\ell,\eta}(x)/(x-\rho^*)$ are interlaced with the zeros of $x \phi_{\ell+2,\eta}(x)$.
\end{enumerate}
\end{theorem}

\begin{proof}
On dividing \eqref{thm:recur} by $x^2 R_{1,\ell,\eta}(1/(2x)) \phi_{\ell+2,\eta}(x)$ and differentiating, we obtain that
\begin{equation}\label{thm:Wronskian}
\begin{multlined}[c][.9\displaywidth]
    \mathcal{W}(x)
    = 2x^4 R_{1,\ell,\eta}^2\rb{\frac{1}{2x}} W\qb{ \phi_{\ell+1,\eta}(x) , x \phi_{\ell+2,\eta}(x) } \\
    + \frac{(\ell+2)^2+\eta^2}{2(\ell+2)^2(2\ell+3)(2\ell+5)}x^2 \phi_{\ell,\eta}^2(x),
\end{multlined}
\end{equation}
where
\begin{equation*}
    \mathcal{W}(x) = W\qb{ \phi_{\ell,\eta}(x) , x^2 R_{1,\ell,\eta}\rb{\frac{1}{2x}} \phi_{\ell+2,\eta}(x) }.
\end{equation*}
Suppose that $\phi_{\ell,\eta}$ and $\phi_{\ell+2,\eta}$ do not have a zero in common. Then obviously $\rho^* \ne x_{\ell,\eta,k}$ for all $k\ge1$ (If not, $\rho^*$ become necessarily the zero of $\phi_{\ell+2,\eta}$, by means of \eqref{thm:recur}).
From \eqref{Wrons1} and Proposition \ref{prop:6.1}, we find that $\mathcal{W}(x)$ is nonnegative on $\R$ and it vanishes only at $x=0$. In other words, the meromorphic function $x^2 R_{1,\ell,\eta}\rb{\frac{1}{2x}} \phi_{\ell+2,\eta}(x) / \phi_{\ell,\eta}(x)$ is increasing on each subintervals of $\R$ partitioned by the zeros of $\phi_{\ell,\eta}(x)$. Since the numerator and denominator have only simple zeros and $x^2 R_{1,\ell,\eta}\rb{1/(2x)} \phi_{\ell+2,\eta}(x)$ shares the zeros with $x(x-\rho^*)\phi_{\ell+2,\eta}(x)$, the zeros of this meromorphic function are interlaced with its poles, which establishes the first result.

In the remaining case when $\rho^* = x_{\ell,\eta,k}$ for some $k\ge1$, we consider $\Phi(x) = \phi_{\ell,\eta}(x)/ R_{1,\ell,\eta}(x)$, which does not vanish at $x=\rho^*$. Accordingly, it can be deduced from \eqref{thm:Wronskian} that $R_{1,\ell,\eta}^{-2}(x)\mathcal{W}(x)$ is nonnegative on $\R$ and it vanishes only at $x=0$. Moreover, we have
\begin{equation*}
    \frac{d}{dx} \frac{x\phi_{\ell+2,\eta}(x)}{\Phi(x)} = \frac{R_{1,\ell,\eta}^{-2}(x)\mathcal{W}(x)}{\Phi^2(x)}.
\end{equation*}
Applying the same argument, the second statement follows as $x\phi_{\ell+2,\eta}(x)/ \Phi(x)$ is increasing on on each subintervals of $\R$ partitioned by the zeros of $\Phi_{\ell,\eta}(x)$.
\end{proof}

As shown in Figure \ref{fig:6.1}, the zeros of $\phi_{\ell,\eta}(x)$ are interlaced with the zeros of $x(x-\rho^*)\phi_{\ell_2,\eta}(x)$ since none of the zeros in $\phi_{\ell,\eta}(x)$ coincide with $\rho^*$.

\begin{figure}[H]
    \centering
    \includegraphics[width=0.495\textwidth, height=0.412\textwidth]{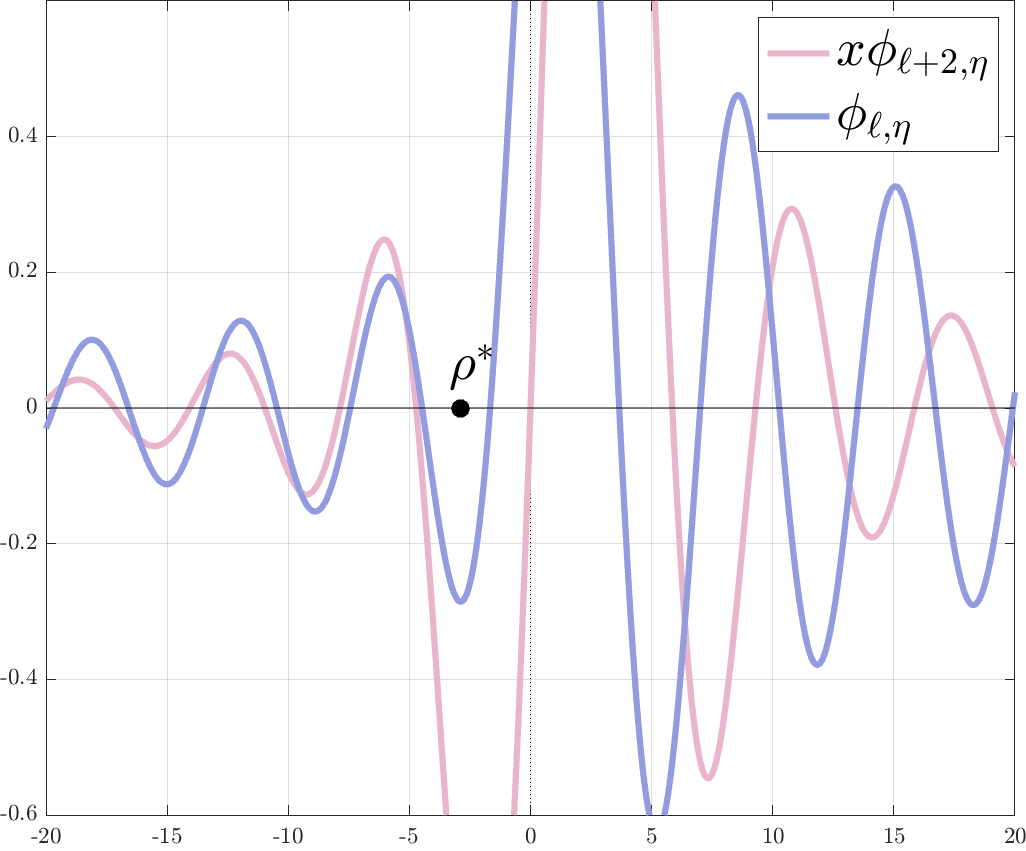}
    \caption{The graphs of $\phi_{\ell,\eta}$ and $x\phi_{\ell+2,\eta}$ when $\eta = 1/2$ and $\ell = -2/5$. The plotted dot denotes the zero of $R_{1,\ell,\eta}(1/(2x))$, i.e., $\rho^* = -(\ell+1)(\ell+2)/\eta$.}
    \label{fig:6.1}
\end{figure}

\bigskip \noindent
{\bf Acknowledgements.}
The author would like to thank Mourad E. H. Ismail for drawing the author's attention to Wimp's work and for many stimulating conversations. 

\printbibliography

\end{document}